\documentclass[12pt]{amsart}
\usepackage{amsrefs}
\usepackage{amsmath,amscd,amssymb}
\usepackage{enumitem}
\usepackage{hyperref}

\newtheorem{theorem}{Theorem}[section]
\newtheorem{prop}[theorem]{Proposition}
\newtheorem{lemma}[theorem]{Lemma}
\newtheorem{cor}[theorem]{Corollary}
\newtheorem{definition}[theorem]{Definition}
\newtheorem{example}[theorem]{Example}

\newtheorem{remark}[theorem]{Remark}

\numberwithin{equation}{section}
\newtheorem*{theorem*}{Theorem}
\newcommand{\bb}[1]{\mathbb{#1}}
\newcommand{\cl}[1]{\mathcal{#1}}

\newcommand{\bs}[1]{\boldsymbol{#1}}

\begin{document}
\title[doubly twisted near-isometries]{Doubly twisted near-isometries: Classification and a Wold-type decomposition}

\author{Sneh Lata}
\address{Department Of Mathematics\\
         School of Natural Sciences\\
        Shiv Nadar Institution Of Eminence\\
        Gautam Buddha Nagar - 201314\\
         Uttar Pradesh, India}
\email{sneh.lata@snu.edu.in}

\author{Santosh Singh Negi}
\address{Department Of Mathematics\\
         School of Natural Sciences\\
        Shiv Nadar Institution Of Eminence\\
        Gautam Buddha Nagar - 201314\\
         Uttar Pradesh, India}
\email{sn210@snu.edu.in}

\author{ Dinesh Singh}
\address{Centre For Digital Sciences\\
      O. P. Jindal Global University\\
   Sonipat\\
        Haryana 131001, India}
\email{dineshsingh1@gmail.com}


\date{}

\maketitle
\begin{abstract}
We introduce and study doubly twisted near-isometries. A doubly twisted near-isometry is a tuple of near-isometries satisfying certain relations determined by a prescribed family of unitaries, thereby generalizing the notion of doubly commuting near-isometries. We establish necessary and sufficient conditions for a tuple of near-isometries to admit a Wold-type decomposition and prove that the existence of such a decomposition automatically ensures its uniqueness by providing an explicit description of the summands. Furthermore, we show that every doubly twisted near-isometry admits a Wold-type decomposition. We also characterize unitary equivalence within the class of doubly twisted near-isometries and construct an analytic model for them. Several examples are included to highlight the distinctions between our results and the corresponding results in the setting of doubly twisted isometries.  
\end{abstract}

\section{Introduction} 
Research related to Wold-type decompositions of operators occupies a central position in operator theory. Originating in the classical work of Wold in the context of stochastic processes, the decomposition became a cornerstone of operator and function theory after the significance of Wold’s result was realized through the work of Halmos on invariant subspaces of the shift operator of infinite multiplicity. We refer to \cite{Hal, Hof}. The Wold decomposition gives a canonical orthogonal splitting of an isometry $T$ on a Hilbert space $\cl H$ into its unitary and pure (shift) parts: 
$$ 
\cl H=\bigoplus_{n=0}^\infty T^n(kerT^*)\bigoplus \bigcap_{n=0}^\infty T^n\cl H.
$$

Note that the subspaces $\cl H_1=\bigoplus\limits_{n= 0}^\infty T^n(kerT^*)$ and $\cl H_2=\bigcap\limits_{n=0}^\infty T^n\cl H$ reduce $T, $ and the restriction of $T$ to $\cl H_2$ is a unitary, and therefore called the unitary part of $T$; whereas the restriction of $T$ to $\cl H_1$ is called its shift part as it is unitarily equivalent to the unilateral shift acting on a vector-valued Hardy space, yielding a concrete functional model for the operator. This structural result is fundamental not only for the analysis of single operators but also for the construction of analytic models and the development of dilation theory. In particular, by way of illustration, we refer to 
\cite{Hal, LMS, lata2018subhardy, Nag, ST, SS} 

Motivated by the wide-ranging impact of the Wold decomposition for isometries, the first and third authors of this paper introduced the notion of a near-isometry in \cite{lata2018subhardy}. A Wold-type decomposition for a near-isometry was obtained implicitly in \cite{lata2018subhardy} and was later established explicitly in \cite{lata2022multivariable}, joint work with Pokhriyal. Although near-isometries relax the strict isometric condition, they retain sufficient structure to admit a meaningful decomposition theory. In Section \ref{model-near}, we use this decomposition, as in the case of an isometry, to derive an analytic model for a near-isometry. To the 
best of our knowledge, this is the first such analytic model. 

S\l{}oci\'nski \cite{Slo} obtained a two-variable analog of the Wold decomposition for doubly commuting isometries. A pair $(T_1, T_2)$ of operators is said to be doubly commuting if $T_1^*T_2=T_2T_1^*$. Every doubly commuting pair of isometries is commuting; however, the converse does not hold in general. In the same paper, he constructed an example of commuting isometries that fail to admit such a decomposition, thereby highlighting the strength of the doubly commutativity. This decomposition was subsequently extended to the $n$-variable setting of doubly commuting isometries by Sarkar \cite{Sar}, and to doubly commuting near-isometries by Lata, Pokhriyal, and Singh \cite{lata2022multivariable}. These multivariable analogues of the Wold decomposition play a key role in the study of invariant subspaces of Hardy spaces over the $n$-torus in both the classical and the de Branges settings \cite{lata2022multivariable, Man, Singh, ST}.  

In recent years, there has been growing interest in doubly twisted variants of tuples of doubly commuting isometries \cite{KRT, Mansi, jaydebmansirakshit2022, sarkar2024orthogonal, pinto}, which generalize the notion of double commutativity in a nontrivial manner. A tuple of isometries $(T_1, \dots, T_n)$ (with $n\ge 2$) on a Hilbert space is said to be a doubly twisted isometry with respect to a commuting family of unitaries $\{U_{ij}: 1\le i<j\le n\}$ if each $T_i$ commutes with every 
$U_{jk}$ and $T_i^*T_j=U_{ij}^*T_jT_i$ for all $1\le i<j\le n$. In a sequence of papers \cite{jaydebmansirakshit2022,sarkar2024orthogonal}, Rakshit, Sarkar, and 
Suryawanshi investigated a Wold-type decomposition for doubly twisted isometries and established a structural and analytic model theory in the strictly isometric setting. Their results demonstrate that the twisted commutation framework is robust enough to support a canonical decomposition and a functional realization. The special case when the unitaries $U_{ij}$ are simply scalar multiples of the identity was treated in \cite{pinto}. More recently, Suryawanshi 
and Solel \cite{Mansi} extended the structural and analytic results of \cite{jaydebmansirakshit2022, sarkar2024orthogonal} to multivariable isometric covariant representations 
associated with product systems of $C^*$-correspondences. 

Motivated by these developments, the present work extends this circle of ideas to the broader class of doubly twisted near-isometries. We demonstrate that the Wold paradigm persists beyond the isometric regime. One of the notable reason for our claim to novelty is that we derive a Wold-type decomposition for this broader class and use it to construct an analytic model for them, identifying the structural mechanisms that compensate for the absence of exact isometry. In doing so, we generalize and refine the known results for doubly twisted isometries \cite{jaydebmansirakshit2022,sarkar2024orthogonal}. 
 
During the preparation of this manuscript, we became aware of \cite{KRT}, where near-isometries are introduced in the context of product systems of $C^*$-correspondences and a Wold-type decomposition is obtained for near-isometric covariant representations, extending the isometric framework of \cite{Mansi} to the near-isometric setting. While \cite{KRT} establishes a Wold-type decomposition in the covariant representation framework, we additionally provide necessary and sufficient conditions for an $n$-tuple of near-isometries to admit such a decomposition. Our results further develop structural features—such as an analytic model and a notion of unitary equivalence—within an explicit operator-theoretic setting. We note that the proof in \cite{KRT} and one of our proofs of the Wold-type decomposition both rely on an induction argument. However, we also present an alternative proof of this decomposition for doubly twisted near-isometries that employs a substantially different set of techniques.

\section{A brief preview}
The rest of the paper is organized as follows: Section \ref{notation} gives notations, basic terminologies, and some preliminary results on near-isometries. 
Section \ref{model-near} develops an analytic model for a near-isometry. Section \ref{char-Wold} introduces the definition of a Wold-type 
decomposition of a tuple of near-isometries and provides necessary and sufficient conditions for an $n$-tuple of near-isometries to admit such a decomposition (Theorem \ref{1 th equivalent condition for wold}). 
Moreover, our proof of Theorem \ref{1 th equivalent condition for wold} yields an explicit description of the summands appearing in the decomposition, which in turn leads to 
the uniqueness of the decomposition (Remark \ref{1 remark unique rep of Wold for n-tuple of near-isometry}). Section \ref{doubly-ex} introduces the notion of a doubly twisted near-isometry and provides an example (Example \ref{1 ex commuting case in near-isometry}) to illustrate the necessity of the additional condition imposed in the definition, as compared with the notion of a doubly twisted isometry. Theorem \ref{doubly-ex1} provides a construction of a doubly twisted near-isometry on a vector-valued Hardy space. Section \ref{doubly-twisted} proves that every doubly-twisted near-isometry admits a Wold-type decomposition (Theorem \ref{1 th wold decom for dtni by induction}). The same section (Remark \ref{cor-isometry}) shows that the Wold-type decomposition for a doubly twisted isometry \cite[Theorem 3.6]{jaydebmansirakshit2022} follows from Theorem \ref{1 th wold decom for dtni by induction} as a special case. 
Section \ref{alternate} contains an alternative proof (Theorem \ref{1th wold decompostion dtni by SOT}) of the Wold-type decomposition of a doubly twisted near-isometry, relying on arguments involving orthogonal projections, in contrast to the proof in Section \ref{doubly-twisted}, which is based on an induction argument. Section \ref{uni-equi} characterizes unitary equivalence within the class of doubly twisted near-isometries. Furthermore, we show that unitary equivalence considered in \cite[Theorem 5.2]{jaydebmansirakshit2022} follows from our equivalence. Through Example \ref{ex-wand}, we demonstrate that the conditions in \cite{jaydebmansirakshit2022} are not sufficient in the broader setting of near-isometries, thereby underscoring the necessity of our theorem. Lastly, Section \ref{doubly-ana} constructs an analytic model for a doubly twisted near-isometry (Theorem \ref{1 th unitary eq of dtni on h}).

\section{Notations and prelimaries} \label{notation}
Let $\bb N$ and $\bb N_0$ denote, respectively, the set of positive and non-negative integers, and let $\bb N_0^n$ be the $n$-fold Cartesian product of $\bb N_0$. To distinguish elements of $\bb N_0^n$ from those of $\bb N_0$, we shall denote them by bold letters such as $\bs k.$ Open unit disc in the complex plane $\bb C$ is denote by $\bb D$, and let $H^2(\bb D)$ denote the Hardy space over $\bb D$. Given $z=(z_1, \dots, z_n)\in \bb C^n$ and ${\bs k} \in \bb N_0^n$, we write $z^{\bs k}$ for $z_1^{k_1}\cdots z_n^{k_n}.$ We shall use  $z^{\bs k}$ interchangeably to denote the complex number $z_1^{k_1}\cdots z_n^{k_n}$ and the function $z\mapsto z^{\bs k},$ the intended meaning will always be clear from the context.  

Recall that the set $\{z^{\bs k}: \bs k\in \bb N_0^n\}$ forms an orthonormal basis for $H^2(\bb D^n).$ Given a Hilbert space $\cl H$, $H^2_{\cl H}(\bb D^n)$ denote the $\cl H$-valued Hardy space over the polydisc $\bb D^n.$ All Hilbert spaces we consider are assumed to be separable. If $\{h_j:j\in \bb N_0\}$ is an orthonormal basis for $\cl H$, then $\{h_jz^{\bs k}: j\in \bb N_0, \bs k\in \bb N_0^n\}$ is an orthonormal basis for $H^2_{\cl H}(\bb D^n).$ Throughout this paper, we shall frequently identify $H^2_{\cl H}(\bb D^n)$ with the Hilbert space tensor product $H^2(\bb D^n)\otimes \cl H$ of $H^2(\bb D^n)$ and $\cl H$. Indeed, the natural map $h_jz^{\bs k}\mapsto z^{\bs k}\otimes h_j$ extends to a unitary from $H^2_{\cl H}(\bb D^n)$ onto $H^2(\bb D^n)\otimes \cl H$. We note that for $\cl H=\bb C$, the Hilbert spaces $\cl H^2_{\cl H}(\bb D^n)$ is simply the Hardy space $H^2(\bb D^n)$ over the polydisc. 

Let $B(\cl H)$ denote the set of all bounded linear operators on $\cl H$; throughout, we shall refer to bounded linear operators simply as operators. One of the most fundamental operators on $H^2(\bb D)$ is $M_z$, the operator of multiplication by the coordinate function $z$, commonly known as the (forward) shift. The shift operator on the vector-valued Hardy space $H^2_{\cl H}(\bb D)$ is given by $M_z \otimes I_{\cl H}$, acting as $z^n\otimes \eta\mapsto z^{n+1}\otimes \eta.$ 
As noted in the introduction, the Wold decomposition of an isometry asserts that the shift part of an isometry is unitarily equivalent to shift on a vector-valued Hardy space. More precisely, 
the shift part of an isometry $T$ is unitarily equivalent to $M_z\otimes I_{kerT^*}$, where the unitary operator is determined by 
$$
 T^n(\eta)\mapsto z^n \otimes \eta, \ \ \ \eta \in kerT^*, \ n\ge 0.
 $$ 

The analogous notion on $H^2(\bb D^n)$ is the tuple $(M_{z_1}, \dots, M_{z_n})$ of multiplication operators, where each $M_{z_i}$ denotes the multiplication by the 
coordinate function $z_i$, that is, $z=(z_1,\dots, z_n)\mapsto z_i$. In analogy with the analytic model for a pure isometry described above, these multiplication 
operators play a central role in an analytic model for a doubly twisted isometry (see \cite{jaydebmansirakshit2022}). In the present paper, we show that these operators continue to be instrumental in obtaining an analytic model for a near-isometry and a doubly twisted near-isometry.

We shall now fix some notations related to tuples of operators. Let $(T_1, \dots, T_n)$ be a tuple of operators on a 
Hilbert space $\cl H$. We fix $I_n$ for the set $\{1, \dots, n\}$, and for any set $A\subseteq I_n$, we define $T_A=(T_{a_1}, \dots, T_{a_m})$ and $T_A^{\bs k} = T_{a_1}^{k_1}\cdots T_{a_m}^{k_m}$, where 
$A=\{a_1, \dots, a_m\}$ and $\bs k=(k_1, \dots, k_m)\in \bb N_0^{m}$. For notational convenience, we shall always assume that the elements of $A$, in the representation of $T_A$ and $T_A^{\bs k}$, are taken in the increasing order, that is, $a_1<\dots<a_m$. Given any positive integer $m$ and $1\le i\le m$, let $e_i\in \bb N_0^{m}$ denotes the $m$-tuple that is $1$ in the $i^{th}$ coordinate and zero else where. Then, for $\bs k\in \bb N_0^{m}$ with $k_i\ge 1$, we have $\bs k-e_i\in \bb N_0^m$ and  
$z^{\bs k -e_i}=z_1^{k_1}\cdots z_i^{k_i-1}\cdots z_m^{k_m}$ for any $z\in \bb C^m.$  

Finally, we give some essential definitions and basic results on near-isometries. 

 \begin{definition}
 A closed subspace $\mathcal{W} \subseteq \cl{H}$ is said to be a wandering subspace for an operator $T\in B(\cl H)$ if $T^i \mathcal{W} \perp T^j \mathcal{W}$ 
 for all non-negative integers $i,j$ with $i \neq j$.
 \end{definition}
 
\begin{definition}
 An operator $T\in B (\mathcal{H})$ is said to be a shift operator on the Hilbert space $\mathcal{H}$ if there exists a wandering subspace $\mathcal{W}$ for $T$ such that 
\begin{equation*}
\mathcal{H}=\bigoplus_{n=0}^{\infty} T^{n}\mathcal{W}.  
\end{equation*}   
\end{definition}

 \begin{definition}
An operator $T\in B(\mathcal{H})$ is said to be a near-isometry if it satisfies the following two conditions:
\begin{enumerate}
\renewcommand{\labelenumi}{(\roman{enumi})}
\item There exists a constant $\delta> 0$ such that $\delta\|x\|\leq \|Tx\|\leq \|x\|$ for all $x\in \mathcal{H}.$
\item For each  $n\geq 0, \ T^{*n}T^{n+1}\mathcal{H}\subseteq T\mathcal{H}$.
\end{enumerate} 
 \end{definition}
 
Condition $(ii)$ in the definition of a near-isometry, namely, $T^{*n}T^{n+1}\mathcal{H}\subseteq T\mathcal{H}$ is equivalent to $T^n(kerT^*)\perp T^{n+1}\cl H$. 
On occasion, we shall use this equivalent formulation of condition $(ii)$. 

The restriction of a near-isometry to an invariant subspace need not remain a near-isometry. We illustrate this with the following example. 

\begin{example}
Let $\cl B$ denote the Bergman space over the open unit disc $\bb D$ in the complex plane. Recall that it consists of holomorphic functions $f(z)=\sum\limits_{n=0}^\infty a_nz^n$ on $\bb D$ such that $\sum\limits_{n=0}^\infty \frac{|a_n|^2}{n+1}<\infty$, and $\left\{\sqrt{n+1}z^n: n\ge 0\right\}$ forms its orthonormal basis.  

It is fairly straight forward to check that $M_z$ is a bounded below contraction on $\cl B.$ Note that $ker(M_z^*)$ is the one-dimensional subspace spanned by the constant function $1$; therefore,  $M_z^n(kerM_z^*)\in span\{z^n\}.$ Further, $\{\sqrt{k+1}z^k: k\ge n+1\}$ forms an orthonormal basis for $M_z^{n+1}\cl B$, which implies that $M_z^n(kerM_z^*)\perp M_z^{n+1}\cl B$. Hence, $M_z$ is a near-isometry on $\cl B.$ 
 
Let $\cl M=\{f\in \cl B: f(1/2)=0\}$. Then $\cl M$ is a closed subspace of $\cl B$. Indeed, $\cl M$ is the orthogonal complement of the one-dimensional subspace of 
$\cl B$ spanned by $k_{\frac{1}{2}}$, the kernel function of $\cl B$ at the point $1/2$. Clearly, $\cl M$ is invariant under $M_z$. For notational sake, let us use $T$ for $M_z$ restricted to $\cl M$. We shall show that $T^{*}T^2(\cl M)\nsubseteq T\cl M.$ Let $f(z)=z-1/2$. Then $T^2(f)=z^3-z^2/2.$  

Now, $T^*(z^3-z^2/2)=P_{\cl M}M_z^*(z^3-z^2/2)=P_{\cl M}(3z^2/4-z/3)$, where $P_{\cl M}$ is the projection onto $\cl M.$ But, 
$P_{\cl M}=I-\langle{\cdot, \widehat{k}_{\frac{1}{2}}}\rangle \widehat{k}_{\frac{1}{2}}$ for $\widehat{k}_{\frac{1}{2}}=k_{\frac{1}{2}}/||k_{\frac{1}{2}}||$. Recall that $k_{\frac{1}{2}}(z)=\frac{1}{(1-\frac{1}{2}z)^2}$. Then, 
 $$
T^*(z^3-z^2/2) = \frac{3}{4}z^2-\frac{1}{3}z-\frac{1}{64}\widehat{k}_{1/2}.
 $$
This imples that $T^*T^2(f)(0)\ne 0$, as $k_{\frac{1}{2}}(0)\ne 0.$ Therefore, $T^*T^2(f)\notin z\cl M=T\cl M$, which establishes that $T=M_z|_ \cl M$ is not a near-isometry.
\end{example}

In contrast to the invariant case, the near-isometric structure is preserved under restriction to a reducing subspace. 

\begin{lemma}\label{1 shift on reducing subspace}
Let $T\in B(\mathcal{H})$ be a near-isometry and let  $\mathcal{M}$ be a reducing subspace of $T$. Then: 
\begin{enumerate}
\item[(i)] $T$ restricted to $\cl M$ is a near-isometry.
\item[(ii)] $T$ is a shift operator if and only if $T|_{\mathcal{M}}$ and $T|_{\mathcal{M}^{\perp}}$ are shift operators.  
\item[(iii)] $T$ is an invertible operator if and only if $T|_{\mathcal{M}}$ and $T|_{\mathcal{M}^{\perp}}$ is an invertible operator. 
\end{enumerate}
\end{lemma}

\begin{proof} Clearly, $T$ restricted to $\cl M$ stays a bounded below contraction. Furthermore, since $\cl M$ reduces $T$, we have $ker(T|_{\cl M})^*=kerT^*\cap \cl M$. Consequently, $(T|_{\cl M})^n(ker(T|_{\cl M})^*)$ is contained in $T^n(kerT^*)$, while $T^{n+1}\cl M$ is contained in $T^{n+1}\cl H$. Thus, it follows that 
$(T|_{\cl M})^n(ker(T|_{\cl M})^*)\perp T^{n+1}\cl M$. Hence, $T$ restricted to $\cl M$ remains a near-isometry.  

To prove $(ii)$, let us first assume that $T$ is a shift operator on $\mathcal{H}$. Then, there exists a wandering subspace $\mathcal{W}$ such that $\mathcal{H}=\bigoplus_{n=0}^{\infty} T^{n}ker T^{*}.$

Since $\mathcal{M}$ and $\mathcal{M}^{\perp}$ reduces $T$, we have $ker T^{*}=ker T^*|_{\mathcal{M}}\oplus ker T^*|_{\mathcal{M}^{\perp}}$.
This implies 
\begin{equation*}
\mathcal{H}=\bigoplus_{n=0}^{\infty}T^{n}(ker T^{*})=\bigoplus_{n=0}^{\infty}T^{n}(ker T^*|_{\mathcal{M}})\bigoplus_{n=0}^{\infty}T^{n}(ker T^*|_{\mathcal{M}^{\perp}}).   
\end{equation*}

Moreover, $\bigoplus_{n=0}^{\infty}T^{n}(ker T^*|_{\mathcal{M}})\subseteq \mathcal{M}$ and $\bigoplus_{n=0}^{\infty}T^{n}(ker T^*|_{\mathcal{M}^{\perp}})\subseteq \mathcal{M}^{\perp}$. Therefore,  
$$\mathcal{M}=\bigoplus _{n\geq 0}(T|_{\mathcal{M}})^{n}ker (T^*|_{\mathcal{M}}) \ \ {\rm and} \ \  \mathcal{M}^{\perp}=\bigoplus _{n\geq 0}(T|_{\mathcal{M}^{\perp}})^{n}ker (T^*|_{\mathcal{M}^{\perp}}),$$ 
which establishes that restriction of $T$ to $\cl M$ and $\cl M^{\perp}$ are both shifts.  

To prove the converse part in $(ii)$, let $\mathcal{W}_{A}$ and $\mathcal{W}_{B}$ be the wandering subspaces of $T|_{\mathcal{M}}$ and $T|_{\mathcal{M}^{\perp}}$ respectively. Then $\mathcal{W}=\mathcal{W}_{A}\bigoplus \mathcal{W}_{B}$ is a wandering subspace for $T$ because $T^{n}\mathcal{W}\perp T^{m}\mathcal{W}$ for all $n\neq m.$ Additionally, $\mathcal{W}_{A}\subseteq \mathcal{W}$, which implies that $\mathcal{M} \subseteq \bigoplus_{n=0}^{\infty}T^{n}\mathcal{W}.$ Similarly $\mathcal{M}^{\perp}\subseteq  \bigoplus_{n=0}^{\infty}T^{n}\mathcal{W}.$ Thus, $\mathcal{H}=\bigoplus_{n=0}^{\infty}T^{n}\mathcal{W}.$  

Now, to settle $(iii)$, note that we can decompose $T$ as 
 $$ T=\begin{bmatrix}  
   T|_{\mathcal{M}} & 0\\
0 & T|_{\mathcal{M}^{\perp}}
\end{bmatrix}$$ from which it is evident that $T$ is invertible if and only if $T|_{\mathcal{M}}$ and $ T|_{\mathcal{M}^{\perp}}$ are invertible.
\end{proof}

\section{An analytic model for near-isometries}\label{model-near}
Just as the Wold decomposition of an isometry yields an analytic model for its shift part, the Wold-type decomposition for a near-isometry (Theorem \ref{1 wold decomposition for a near-isometry}) also leads to an analytic model of its shift component, which appears to be new in the literature. In particular, we show that the shift part of a near-isometry is unitarily equivalent to an operator-valued  weighted shift acting on a vector-valued Hardy space $H^2_{\cl H}(\bb D)$. To the best of our knowledge, this provides the first analytic model for a near-isometry. To state and prove this result precisely, we first recall the Wold-type decomposition for a near-isometry from \cite{lata2022multivariable}, and then the notion of an operator-valued weighted shift.

\begin{theorem}\label{1 wold decomposition for a near-isometry}
 Let $T\in B(\mathcal{H})$ be a near-isometry. Then $\mathcal{H}$ can be decomposed into reducing subspaces of $T$ as $\mathcal{H}=\mathcal{H}_{T,S}\bigoplus \mathcal{H}_{T,I}$ such that $T|_{\mathcal{H}_{T,S}} $ is a shift operator and $T|_{\mathcal{H}_{T,I}} $ is an invertible operator. Moreover, $\mathcal{H}_{T,S}=\bigoplus_{n=0 }^{\infty}T^{n}(ker T^{*})$ and $\mathcal{H}_{T,I}=\bigcap_{n\geq 0}T^{n}\mathcal{H}$ where $\mathcal{W}=\mathcal{H}\ominus T\mathcal{H}$.
\end{theorem}

Henceforth, if $T$ is a near-isometry on a Hilbert space $\cl H$, we shall use $\cl H_{T, S}$ and $\cl H_{T, I}$ to denote the reducing subspaces for $T$ given by Theorem \ref{1 wold decomposition for a near-isometry}, with the meaning as explained in that theorem. We shall refer to the restriction of $T$ to $\cl H_{T, S}$ and to $\cl H_{T, I}$ as the shift part and the invertible part, respectively.

Let $\cl H$ be a Hilbert space and let $\{T_n\}_{n\ge 0}$ be a sequence of operators on $\cl H$ with $\sup_{n}||T_n||<\infty$. The operator-valued weighted shift operator associated with the weight sequence $\{T_n\}_{n\ge 0}$ is the operator $S:H^2_{\cl H}(\bb D)\to H^2_{\cl H}(\bb D)$ defined by 
$$
S(z^n\otimes x)= z^{n+1}\otimes T_n x.
$$

It is fairly straightforward to verify that 
$$S^k(z^n\otimes x)=z^{n+k}\otimes (T_{n+k-1}\cdots T_n)(x_n), 
$$ which yields that $S^n(1\otimes x)\perp S^k(1\otimes y)$ for all $x, y\in \cl H$, whenever $n\ne k$. Consequently, if each $T_n$ is surjective, then $S$ is a shift operator with wandering subspace $\cl H.$  

The following result shows that the shift part of a near-isometry is unitarily equivalent to an operator-valued weighted shift, and thereby provides an analytic model for it. 

\begin{theorem}\label{1 th analytic model of shift ni} 
Let $T\in B(\cl H)$ be a near-isometry. Then $T$ is a shift operator if and only if $T$ is unitarily equivalent to an operator-valued weighted shift $S$ with associated weight sequence $\{T_n\}_{n\ge 0}$ of invertible operators in $B(\cl H)$ that, in addition, satisfy 
\begin{equation}\label{lowerbd}
c ||x|| \le ||T_nx||\le ||x|| 
\end{equation}
for all $x\in \cl H$ and some $c>0$.
\end{theorem}

\begin{proof} Let $T$ be a shift operator on $\cl H$. Then, by the Wold-type decomposition of a near-isometry, we have 
\begin{equation*}
\mathcal{H}=\bigoplus_{n\geq 0}T^{n}kerT^{*}.    
\end{equation*}

For fixed $n\ge 0$, the operator $T^n: kerT^*\to T^n(kerT^*)$ is an invertible operator; therefore, its polar decomposition yields a unitary map 
$\Lambda_{n}:ker T^{*}\to T^{n}(ker T^{*})$. Consequently, the map $U:\cl H\to H^2_{kerT^*}(\bb D)$ given by 
$$
U(T^nx)= z^n\otimes \Lambda_n^* T^n x
$$
is a well-defined isometry. Furthermore, for $x\in kerT^*$ we can write $\Lambda_n(x)=T^ny$  some $y\in kerT^*$, which yields that 
$U(T^ny)=z^n\otimes \Lambda_n^*T^ny=z^n\otimes x$. This establishes that $U$ is onto; hence, $U$ is a unitary. Now, consider   
\begin{equation*}
UTU^{*}(z^{n}\otimes x)=UT\Lambda_{n}x=z^{n+1}\otimes \Lambda_{n+1}^*T\Lambda_nx.    
\end{equation*}
 
Setting $T_n=\Lambda_{n+1}^*T\Lambda_n$, we obtain that $T_n\in B(kerT^*)$ is invertible and $||T_nx||=||T\Lambda_nx||$ for all $x\in kerT^*.$ Now, since $T$ is a near-isometry and each $\Lambda_n$ is a unitary, we conclude that each $T_n$ is a contraction and the sequence $\{T_n\}$ is uniformly bounded below. This completes the 
proof of the forward implication.

To prove the converse, let $\{T_n\}_{n\ge 0}$ be a sequence of invertible operators on $\cl H$ satisfying (\ref{lowerbd}) and $S$ be the operator-valued weighted shift on $H^2_{\cl H}(\bb D)$ with weights $\{T_n\}$. Since $T$ is unitarily equivalent to $S$, proving $S$ is a near-isometry automatically implies that $T$ is a near-isometry. To this end, we show that $S$ is a near-isometry. Let $x=\sum\limits_{i=0}^\infty z^{i}\otimes x_{i}\in H^{2}_{\cl H}(\bb D).$ Then,  
$$
\|S(x)\|^2= \sum_{i=0}^\infty \left\|z^{i+1}\otimes T_ix_{i}\right\|^{2}=\sum_{i=0}^\infty\| T_{i}x_{i}\|^{2}.  
$$

Thus, $c||x||\le ||Sx||\le ||x||$ for $x\in H^2_{\cl H}(\bb D)$, where $c$ is the lower bound for every $T_n$.

Further, $S^n(\cl H)=span{\{z^n\otimes x: x\in \cl H\}}$ and $S^{n+1}H^2_{\cl H}(\bb D)=z^{n+1}H^2_{\cl H}(\bb D).$ It follows that $S^n(\cl H)\perp S^{n+1}H^2_{\cl H}(\bb D)$. 
Consequently, $S^{*n}S^{n+1}(H^2_{\cl H}(\bb D))\subseteq S(H^2_{\cl H}(\bb D))$. Indeed, since each $T_n$ is invertible, the orthogonal complement of the range of S equals 
$\cl H$, identified with the subspace $\{1\otimes x: x\in \cl H\}$. Hence $S$ is a near-isometry on $H^2_{\cl H}(\bb D),$ which implies that $T$ is a near-isometry on $\cl H$. Finally, since $S$ is a shift on $H^2_{\cl H}(\bb D)$, it follows that $T$ itself is a shift. This completes the proof.
\end{proof}

\section{A Wold-type decomposition for an $n$-tuple of near-isometries}\label{char-Wold}
In this section, we first define what we mean by a Wold-type decomposition for an $n$-tuple of near-isometries. Then we  establish–in Theorem \ref{1 th equivalent condition for wold}–two sets of conditions under which an $n$-tuple of near-isometries admits this decomposition. These conditions, in fact, determine the exact form of the summands and thereby yield the uniqueness of such a decomposition whenever it exists (Remark \ref{1 remark unique rep of Wold for n-tuple of near-isometry}). 

 \begin{definition}\label{1 de Wold decom for a near-isometry}
Let $\mathcal{H}$ be a Hilbert space. An n-tuple of near-isometries $T=(T_{1},T_{2},\dots,T_{n})$ on $\mathcal{H}$ is said to admit a Wold-type decomposition if  there exist $2^{n}$ closed subspaces $\{\mathcal{H}_{A}\}_{A\subseteq I_{n}}$ reducing each $T_{i}$ such that   
\begin{enumerate}
\renewcommand{\labelenumi}{(\roman{enumi})}
\item$\mathcal{H}=\bigoplus\limits_{A\subseteq{I_{n}}} \mathcal{H}_{A},$
\item $T_{i}|_{\mathcal{H}_{A}}$ is a shift operator if $i\in A,$
\item $T_{i}|_{\mathcal{H}_{A}}$ is an invertible operator if $i\in I_{n}\setminus A.$
\end{enumerate}
\end{definition}

\begin{theorem}\label{1 th equivalent condition for wold}
Let $T=(T_{1},T_{2},\dots,T_{n})$ be an $n$-tuple of near-isometries on $\mathcal{H}.$ 
Then the following statements are equivalent:
\begin{enumerate}
\renewcommand{\labelenumi}{(\roman{enumi})}
\item T admits a Wold-type decomposition.
\item $\mathcal{H}_{T_{i},S}$ reduces $T_{k}$ for all $i,k\in I_{n}.$
\item $\mathcal{H}_{T_{i},I}$ reduces $T_{k}$ for all $i,k\in I_{n}.$
\end{enumerate}
\end{theorem}

\begin{proof} The equivalence of $(ii)$ and $(iii)$ follows from the Wold-type decomposition for a near-isometry (Theorem \ref{1 wold decomposition for a near-isometry}). To complete the proof, we shall show that the condition $(i)$ is equivalent to the condition $(ii)$.  

We shall first show that $(i)$ implies $(ii)$. Suppose $T$ admits a Wold-type decomposition. Then, we can decompose  
\begin{equation}\label{1 eq1 equivalent condition for wold}
\mathcal{H}=\bigoplus_{A\subseteq{I_{n}}}\mathcal{H}_{A},
\end{equation}
where $\cl H_{A}$ are closed subspaces of $\cl H$ such that $T_{j}$ is a shift operator on $\cl H_A$ if $j\in A$ and an an invertible operator on $\cl H_A$ 
whenever $j\in I_{n}\setminus A.$  

Fix $i\in I_{n}$, and set  
$$
X_{i}=\bigoplus_{A\subseteq{I_{n}},i\in A } \mathcal{H}_{A} \ \ {\rm and} \ \ Y_{i}=\bigoplus_{A\subseteq{I_{n}},i\not\in A } \mathcal{H}_{A}.
$$

This allows us to rewrite (\ref{1 eq1 equivalent condition for wold}) as
$\mathcal{H}=X_{i}\bigoplus Y_{i}$, where  

\begin{itemize}
\item $X_{i}$ and $Y_{i}$ both reduce $T_{k}$ for all $k\in I_{n}$, because $\mathcal{H}_{A}$ reduces $T_{k}$ for every $A\subseteq{I_{n}}$;
\item $T_{i}|_{X_{i}}$ is a shift operator using Lemma \ref{1 shift on reducing subspace}, because $T_{i}|_{\mathcal{H}_{A}}$ is a shift operator for every $A$ such that $i\in A$;
 \item $T_{i}|_{Y_i}$ is an invertible operator using Lemma \ref{1 shift on reducing subspace}, because $T_{i}|_{\mathcal{H}_{A}}$ is an invertible operator for every $A$ such that $i\notin A$. 
\end{itemize}

Then, by the uniqueness of the Wold-type decomposition for a near-isometry, we have $X_{i}=\mathcal{H}_{T_{i},S}$ and $Y_{i}=\mathcal{H}_{T_{i},I}$. Thus, 
$\mathcal{H}_{T_{i},S}$ reduces $T_{k}$ for all $k\in I_{n}$. Hence, $(ii)$ holds. 

We shall now prove that $(ii)$ implies $(i)$. Suppose $\mathcal{H}_{T_{i},S}$ reduces $T_{k}$ for all $i,k\in I_{n}$. Then, $\mathcal{H}_{T_{i},I}$ also reduces $T_{k}$ for all $i,k\in I_{n}$. Before proceeding further, we fix a few notations. For a non-empty subset $B$ of $I_n$ and $A\subseteq B$, we define
$$
 \mathcal{H}_A^B=\left(\bigcap_{i\in B\setminus A}\mathcal{H}_{T_{i},I}\right)\bigcap \left(\bigcap_{k\in A}\mathcal{H}_{T_{k},S}\right),  
$$
where $\mathcal{H}_{\emptyset}^B=\bigcap_{i\in B}\mathcal{H}_{T_{i},I}$ and $\mathcal{H}_{B}^{B}=\bigcap_{i\in B}\mathcal{H}_{T_{i},S}$.  

To prove $(i)$, we first establish that 
\begin{equation}\label{obs1}
\mathcal{H}_{A}^{I_{m}}\subseteq \mathcal{ H}_{A\cup \{j\}}^{I_{m}\cup \{j\}}\oplus \mathcal{H}_{A}^{I_{m}\cup \{j\}} 
\end{equation}
whenever $m<j<n$ and $A\subseteq I_m$. Note that the direct sum on the right hand side is well-defined, since by definition 
$\mathcal{ H}_{A\cup \{j\}}^{I_{m}\cup \{j\}}$ is a subset of $\cl H_{T_j, S}$ and $\mathcal{H}_{A}^{I_{m}\cup \{j\}}$ is a subset of $\cl H_{T_j,I}$, and these two subspaces are orthogonal. 

To verify \ref{obs1}, let $m<j<n$ and $A\subseteq I_m$. Then $\mathcal{H}_{A}^{I_{m}}$ reduces $T_{j}$, since $\mathcal{H}_{T_{i},S},\mathcal{H}_{T_{i},I}$ reduces $T_{j}$ for all $i\in I_{n}$, which implies that $T_j$ is a near-isometry on $\cl H_A^{I_m}$. Therefore, using the Wold-type decomposition for a near-isometry, we decompose $\cl H_A^{I_m}$ as  

$$
\mathcal{H}_{ A}^{I_{m}}=\bigoplus_{\ell_{j}\geq 0} T_{j}^{\ell_{j}}\left(ker (T_{j}^{*}|_{\mathcal{H}_{A}^{I_{m}}})\right)\bigoplus \bigcap_{\ell_{j}\geq 0}T_{j}^{\ell_{j}}\mathcal{H}_{A}^{I_{m}}.
$$

Further, since $\mathcal{H}_{A}^{I_{m}}$ reduces $T_{j}$, $ker (T_{j}^{*}|_{\mathcal{ H}_{A}^{I_{m}}})=ker T_{j}^{*}\bigcap \mathcal{H}_{A}^{I_{m}}.$ Then
\begin{equation}\label{1 eq3 equivalent condition for wold}
\mathcal{H}_{A}^{I_{m}}=\bigoplus_{\ell_{j}\geq 0} T_{j}^{\ell_{j}}\left(ker T_{j}^{*}\bigcap \mathcal{H}_{A}^{I_{m}}\right)\bigoplus \bigcap_{\ell_{j}\geq 0}T_{j}^{\ell_{j}}\mathcal{H}_{A}^{I_{m}}.
\end{equation}

For $\ell_{j}\geq 0$, we have 
$$
T_{j}^{\ell_{j}}(ker T_{j}^{*}\bigcap \mathcal{ H}_{A}^{I_{m}})\subseteq T_{j}^{\ell_{j}}(ker T_{j}^{*}),
$$
and also 
$$
T_{j}^{\ell_{j}}\left(ker T_{j}^{*}\bigcap \mathcal{H}_{A}^{I_{m}}\right)\subseteq T_{j}^{\ell_{j}}(\mathcal{H}_{A}^{I_{m}})\subseteq\mathcal{H}_{A}^{I_{m}}.
$$
Therefore, 
\begin{equation}\label{1 eq7 equivalent condition for wold}
\bigoplus_{\ell_{j}\geq 0} T_{j}^{\ell_{j}}\left(ker T_{j}^{*}\bigcap \mathcal{H}_{A}^{I_{m}}\right)\subseteq\mathcal{H}_{T_{j},S}\bigcap \mathcal{H}_{A}^{I_{m}}=\mathcal{H}_{A\cup \{j\}}^{I_{m}\cup \{j\}}.
\end{equation}

Additionally, 
\begin{equation}\label{1 eq8 equivalent condition for wold}
\bigcap_{\ell_{j}\geq 0}T_{j}^{\ell_{j}}(\mathcal{H}_{A}^{I_{m}})\subseteq \mathcal{H}_{A}^{I_{m}}\bigcap\left(\bigcap_{\ell_{j}\geq0}T_{j}^{\ell_{j}}\mathcal{H}\right)
=\cl H_{A}^{I_m}\cap \cl H_{T_j, I}=\cl H_{A}^{I_m\cup \{j\}}.
\end{equation}

Lastly, using (\ref{1 eq7 equivalent condition for wold}) and (\ref{1 eq8 equivalent condition for wold}) in (\ref{1 eq3 equivalent condition for wold}), we obtain the desired conclusion   
\begin{equation}\label{1 eq9 equivalent condition for wold}
\mathcal{H}_{A}^{I_{m}}\subseteq \mathcal{ H}_{A\cup \{j\}}^{I_{m}\cup \{j\}}\oplus \mathcal{H}_{A}^{I_{m}\cup \{j\}}.
\end{equation}

Finally, to prove $(i)$, we shall show that 
$$
 \mathcal{H}=\bigoplus_{A\subseteq I_{n}}\mathcal{H}_{A}^{I_{n}},  
$$
where, as defined above,  
$$
\mathcal{H}_{A}^{I_n}=\left(\bigcap_{i\in I_{n}\setminus A}\mathcal{H}_{T_{i},I}\right)\bigcap \left(\bigcap_{k\in A}\mathcal{H}_{T_{k},S}\right).
$$

We shall use induction to prove the claim. Since, $T_1$ is a near-isometry on $\cl H$, we can write $\cl H = \cl H_{T_1, I}\oplus \cl H_{T_1, S},$
which, as per the notations defined above, is same as writing 
$$
\cl H =\cl H_{\emptyset}^{1}\oplus \cl H_{\{1\}}^{\{1\}}.
$$
Using (\ref{1 eq9 equivalent condition for wold}), we obtain $\mathcal{H}_{\{1\}}^{\{1\}}\subseteq \mathcal{H}_{\{1\}}^{\{1,2\}}\oplus \mathcal{H}_{\{1,2\}}^{\{1,2\}}$ and $\mathcal{H}_{\emptyset}^{\{1\}}\subseteq \mathcal{H}_{\emptyset}^{\{1,2\}}\oplus \mathcal{H}_{\{2\}}^{\{1,2\}}$.
Therefore, we have  
$$
\mathcal{H}=\cl H_{\emptyset}^{\{1,2\}} \oplus \cl H_{\{1\}}^{\{1,2\}}\oplus \cl H_{\{2\}}^{\{1,2\}}\oplus \mathcal{H}_{\{1,2\}}^{\{1,2\}}.
$$
Continuing in this same manner, suppose we obtain  
\begin{equation}\label{obs2}
\cl H=\bigoplus_{A\subseteq I_m} \cl H_A^{I_m}.
\end{equation}
for some $m< n$. Now, again using (\ref{1 eq9 equivalent condition for wold}), we get $\mathcal{H}_{A}^{I_m}\subseteq \mathcal{H}_{A}^{I_{m+1}}\oplus \mathcal{H}_{A\cup \{m+1\}}^{I_{m+1}}$, for each subset $A$ of $I_m$. Therefore,  
$\mathcal{H}\subseteq \bigoplus\limits_{A\subseteq I_{m+1}}\mathcal{H}_{A}^{I_{m+1}}.$ Thus, by induction, we obtain that 
\begin{equation}\label{obs3}
\cl H= \bigoplus_{A\subseteq I_n}\cl H_A^{I_n}.
\end{equation} 
Lastly, recall that $\mathcal{H}_{A}^{I_n}=\left(\bigcap_{i\in I_{n}\setminus A}\mathcal{H}_{T_{i},I}\right)\bigcap \left(\bigcap_{k\in A}\mathcal{H}_{T_{k},S}\right).$ Then  $\cl H_A^{I_n}\subseteq \cl H_{T_i,S}$ if $i\in A$ and $\cl H_A^{I_n}\subseteq \cl H_{T_i, I}$ if $i\notin A$. Also, $\cl H_A^{I_n}$ reduces $T_i$ for each $i$. 
Therefore, by Lemma \ref{1 shift on reducing subspace}, $T_{i}|_{\mathcal{H}_{A}^{I_n}}$ is a shift operator if $i\in A$ and $T_{i}|_{\mathcal{H}_{A}^{I_n}}$ is an invertible operator if $i\notin A$. Hence, the decomposition given by (\ref{obs3}) is a Wold-type decomposition for the tuple $(T_1, \dots, T_n)$. This completes the proof.
\end{proof}

\begin{remark}\label{1 remark unique rep of Wold for n-tuple of near-isometry}
If $T=(T_1,\dots, T_{n})$ is an $n$-tuple of near-isometries on $\cl H$ admitting a Wold-type decomposition, then the uniqueness of the decomposition follows from the proof of Theorem \ref{1 th equivalent condition for wold}. Indeed, the proof reveals that 
$$
\cl H=\bigoplus_{A\subseteq I_n} \cl H_A,
$$
where
$$
\mathcal{H}_{A}=\left(\bigcap_{i\in I_{n}\setminus A}\mathcal{H}_{T_{i},I}\right)\bigcap \left(\bigcap_{k\in A}\mathcal{H}_{T_{k},S}\right),
$$
with $\mathcal{H}_{T_{i},S}$ and $\mathcal{H}_{T_{i},I}$ the unique subspaces given by the Wold-type decomposition of the near-isometry $T_i.$  
\end{remark}

\begin{remark} Furthermore, the summands $\cl H_A$ are maximal with respect to the property that they reduce each $T_i$, and the operator $T_i$ is a shift on $\cl H_A$ if $i\in A$, and invertible on $\cl H_A$ otherwise. Indeed, suppose $\cl M$ is a subspace of $\mathcal{H}$ that reduces $T_i$ and $T_i$ is a shift on $\cl M$. Then, $T_i$ is a near-isometry on $\cl M$; therefore, by the Wold-type decomposition of a near-isometry,
$$
\cl M=\bigoplus\limits_{k\ge0}T_i^k\left(ker(T_i|_{\cl M})^*\right).
$$
However, $\left(T_i|_{\cl M}\right)^*=T^*_i|_{\cl M}$, which implies that $ker(T_i|_{\cl M})^*=kerT_i^*\cap \cl M$. Thus, we have 
$\cl M\subseteq \bigoplus\limits_{k\ge0}T_i^k (kerT_i^*)=\cl H_{T_i, S}$. Similarly, if $T_i$ is invertible on $\cl M$, we can use the Wold-type decomposition of a 
near-isometry to conclude $\cl M\subseteq \cl H_{T_i, I}$.  
\end{remark}

\section{Doubly twisted near-isometries: definition and a construction}\label{doubly-ex}
In this section, we introduce the notion of a doubly twisted near-isometry and describe its construction (Theorem \ref{doubly-ex1}) from a collection of unitaries and near-isometries.

\begin{definition}\label{1 def of doubly twisted near-isometry}
Let $n\geq 2$. Let $\{U_{ij}\}_{1\leq i<j\leq n}$ be $\binom{n}{2}$ commuting unitaries on $\mathcal{H}$. An $n$-tuple $T=(T_{1},T_{2},\dots,T_{n})$ of near-isometries  on $\mathcal{H}$ is said to be a doubly twisted near-isometry with respect to $\{U_{ij}: 1\le i<j\le n\}$ if for all $1\le i<j\le n$ and $1\le k\le n$, we have
\begin{enumerate}
\renewcommand{\labelenumi}{(\roman{enumi})}
\item $T_{i}^{*}T_{j}=U_{ij}^{*}T_{j}T_{i}^{*},$
\item $T_{k}U_{ij}=U_{ij}T_{k},$
\item $T_{i}T_{j}=U_{ij}T_{j}T_{i}$.
\end{enumerate}
\end{definition}

Henceforth, when referring to a doubly twisted near-isometry, we shall not explicitly mention the associated unitaries $\{U_{ij}: 1\le i<j\le n\}$ whenever they are clear from the context.

\begin{remark} Suppose that, for a doubly twisted near-isometry, we define $U_{ji}:=U_{ij}^{*}$ when $i<j$. Then one can easily show that the conditions $(i)-(iii)$ in the definition in fact holds for all $1\le i,j,k\le n$ whenever $i\ne j$.
\end{remark}

\begin{remark} If the near-isometries in the definition of a doubly twisted near-isometry are in fact isometries, then condition $(iii)$ follows from conditions 
$(i)$ and $(ii)$ \cite[Lemma 3.1]{jaydebmansirakshit2022}. However, this implication fails in the near-isometric setting, as demonstrated by the example below. This distinction becomes significant in Theorem \ref{1 th wold decom for dtni by induction}, where condition $(iii)$ is essential for obtaining a Wold-type decomposition. Moreover, the example shows that a pair failing condition $(iii)$ need not admit a Wold-type decomposition, thereby reinforcing the necessity of condition $(iii)$ in the near-isometric setting.
\end{remark}

\begin{example}\label{1 ex commuting case in near-isometry}
Let $\mathcal{H}=\mathbb{C}\oplus H^2(\mathbb{D})$. For $\phi(z)=\frac{1}{6+3z}$, the Toeplitz operator $M_{\phi}$ is invertible on $H^2(\bb D)$ and 
$\|M_{\phi}\|=||\phi||_{\infty}=\frac{1}{3}$. Further, let $f=\frac{1}{2(1-rz)}$ for some positive $r$ such that $r^2\le \frac{7}{16}$ . Then $f\in H^2(\bb D);$ therefore, 
$F(g) = \langle g, f \rangle$ is a bounded linear functional on $H^2(\bb D)$ and $||F||=||f|| = \frac{1}{2\sqrt{1-r^2}}.$  We shall now define two operators on $\cl H$ as follows. Let 
$$
T_{1}=\begin{bmatrix}
r & F\\
0 & M_{\phi}    
\end{bmatrix}^*
\ \ \ {\rm and} \ \ \ 
T_{2}=\begin{bmatrix}
r & 0\\
0 & M_{z}
\end{bmatrix}.
$$ 
We first show that $T_1$ and $T_2$ are near-isometries on $\cl H$. It is immediate that $T_2$ is a near-isometry, since $M_z$ is an isometry on $H^2(\bb D)$. Now, to show  that $T_1$ is a near-isometry, we prove that it is an invertible contraction. This follows once we establish that $T_1^*$ is an invertible contraction. To see this, first note that that $M_\phi$ is invertible on $H^2(\bb D)$ and $r>0$. It follows that $T_1^*$ is both injective and surjective, and hence, $T_1^*$ is invertible. We now prove that $T_1^*$ is a contraction.  
$$
T_1T_1^* =  
\begin{bmatrix}
r & F\\
0 & M_{\phi}    
\end{bmatrix}^*
\begin{bmatrix}
r & F\\
0 & M_{\phi}    
\end{bmatrix}
=
\begin{bmatrix}
r^2 & rF\\
rF^* & F^*F+M_\phi^*M_{\phi}    
\end{bmatrix}.
$$ 
Then for $\lambda\in \mathbb{C}$ and $h\in H^2(\bb D)$, we have
\begin{eqnarray*}
\langle{T_1T_1^*(\lambda\oplus h), \lambda\oplus h}\rangle &=& r^2|\lambda|^2+2rRe(\bar{\lambda}\langle{h,f}\rangle)+|\langle{h,f}\rangle|^2 +||M_\phi h||^2\\
&\le & r^2|\lambda|^2+(r^2|\lambda|^2+|\langle{h,f}\rangle|^2)+|\langle{h,f}\rangle|^2+||\phi||^2_\infty||h||^2\\
&\le & 2r^2|\lambda|^2 +2||h||^2||f||^2+\frac{1}{9}||h||^2\\
&=& 2r^2|\lambda|^2+\left(\frac{1}{2(1-r^2)}+\frac{1}{9}\right)||h||^2\\
&\le &|\lambda|^2+||h||^2,
\end{eqnarray*}
since our assumption $r^2\le \frac{7}{16}<\frac{1}{2}$ implies that $2r^2<1$ and $\frac{1}{2(1-r^2)}+\frac{1}{9}\le 1$. Thus, $T_1^*$ is a contraction. Hence, we conclude that 
$T_1$ is an invertible contraction, which establish that $T_1$ is a near-isometry.  

Next, we show that $T_1^*T_2=T_2T_1^*$, yet $T_1T_2\ne T_2T_1.$ To establish it, observe that 
\begin{equation*}
T_{1}^*T_{2}=
\begin{bmatrix}    
 r^2 & FM_{z}\\
0 & M_{\phi}M_{z}
\end{bmatrix}
\ \ \ {\rm and} \ \ \
T_{2}T_{1}^*=
\begin{bmatrix}
r^2 & rF\\
0 & M_{z}M_{\phi}
\end{bmatrix}.
\end{equation*}
But, by the choice of $f$, we have $rF=FM_z$. Also, $M_\phi M_z=M_zM_{\phi}$. Therefore, $T_1^*T_2=T_2T_1^*$. Now to show $T_1T_2\ne T_2T_1,$ consider
\begin{equation*}
T_{1}T_{2}=
\begin{bmatrix}    
 r^2 & 0\\
rF^* & M_{\phi}^*M_{z}
\end{bmatrix}
\ \ \ {\rm and} \ \ \
T_{2}T_{1}=
\begin{bmatrix}
r^2 & 0\\
M_zF^* & M_{z}M_{\phi}^*
\end{bmatrix}.
\end{equation*}
Clearly, since $\phi$ is non-constant bounded analytic function, $M_\phi^*$ can't commute with $M_z$, which establishes that $T_1T_2\ne T_2T_1.$ Hence, we have near-isometries 
$T_1$ and $T_2$ that satisfy conditions (i) and (ii) of Definition \ref{1 def of doubly twisted near-isometry}, yet fails condition $(iii)$.  

Finally, we use Theorem \ref{1 th equivalent condition for wold} to demonstrate that $(T_1, T_2)$ does not admit a Wold-type decomposition. Recall that it suffices to show that 
$\mathcal{H}_{T_{2},I}$ is not invariant under $T_{1}$. It is fairly straightforward to verify $H_{T_2, I}=\bb C\oplus \{0\}$. For $1\oplus 0\in H_{T_2, I}$, note that 
\begin{equation}\label{1 check wold}
T_1(1\oplus 0)=
\begin{bmatrix}
r& 0\\
F^* & M_{\phi}^{*}
\end{bmatrix}
\begin{bmatrix}
1\\
0
\end{bmatrix}
=
\begin{bmatrix}
r \\
F^{*}(1)
\end{bmatrix}=
\begin{bmatrix}
r \\
f
\end{bmatrix},
\end{equation}
which does not belong $\cl H_{T_2,I}$ as $f\ne 0.$ Thus, $\mathcal{H}_{T_{2},I}$ is not invariant under $T_{1}$; hence, by Theorem \ref{1 th equivalent condition for wold}, $T$ does not admit a Wold-type decomposition.
\end{example}

If $T$ is a bounded below operator on a Hilbert space $\cl H$, then we fix the notation $T^{\#}$ for its left inverse $(T^*T)^{-1}T^*$. The following are some simple yet important properties of a doubly twisted near-isometry. 

\begin{lemma} \label{ 1 le consequence of definiton of doubly twisted near-isometry}
Let $T=(T_{1},\dots, T_{n})$ be a doubly twisted near-isometry. Then for $1\le i,j, k\le n$ with $i\neq j$, we have
 \begin{enumerate}
 \renewcommand{\labelenumi}{(\roman{enumi})}
 \item $T_{k}^{*}U_{ij}=U_{ij}T_{k}^{*}$ 
 \item $T_{k}^{\#}U_{ij}=U_{ij}T_{k}^{\#}$
 \item $U_{ij}=T_{i}^{\#}T_{j}^{\#}T_iT_j$
 \end{enumerate}
\end{lemma}

\begin{proof} By definition, $T_{k}U_{ij}=U_{ij}T_{k}$ for all $i,j,k\in I_n$ with $i\neq j$; therefore the adjoint of $T_k$ also commute with $U_{ij}$, which proves part 
$(i)$.
 
To prove $(ii)$, first note that $T_{k}^{*}T_{k}U_{ij}=U_{ij}T_{k}^{*}T_{k}$. Then, 
$$
(T_{k}^{*}T_{k})^{-1}U_{ij}=U_{ij}(T_{k}^{*}T_{k})^{-1},     
$$
which yields
$$
T_{k}^{\#}U_{ij}=(T_{k}^{*}T_{k})^{-1}T_{k}^{*}U_{ij}=(T_{k}^{*}T_{k})^{-1}U_{ij}T_{k}^{*}=U_{ij}(T_{k}^{*}T_{k})^{-1}T_{k}^{*}=U_{ij}T_{k}^{\#}.    
$$
Lastly, to prove $(iii)$, note that $T_{i}T_{j}=U_{ij}T_{j}T_{i}$, which implies that $T_{j}T_{i}U_{ij}=T_{i}T_{j}$. Since $T_i^{\#}$ and $T_j^{\#}$ respectively are the left inverses of $T_i$ and $T_j$, therefore we conclude that $U_{ij}=T_{i}^{\#}T_{j}^{\#}T_iT_j$, which completes the proof.
\end{proof}

Note that the commutativity of $\{U_{ij}\}_{1\leq i, j\leq n}$ was not used in the proof of Lemma \ref{ 1 le consequence of definiton of doubly twisted near-isometry}. 
Indeed, commutativity follows directly from parts $(ii)$ and $(iii)$ of the lemma.

We shall now review the construction of a doubly twisted isometry introduced in \cite{jaydebmansirakshit2022}. We then make a few observations concerning this construction, which motivate and lead to our own construction of a doubly twisted near-isometry.

\begin{definition}
Let $\mathcal{H}$ be a Hilbert space and let $U$  be a unitary operator on $\mathcal{H}$. Then, for $1 \leq j \leq n, \ D_{j}[U]$ denote the diagonal unitary operator on $H^2_{\cl H}(\bb D^n)$ that is given by 
$$
z^{\bs k}\eta \mapsto z^{\bs k} U^{k_j}\eta,
$$
where $k = (k_{1}, k_{2}, \dots, k_{n}) \in \mathbb{N}_{0}^{n}$ and $\eta \in \mathcal{H}$.
 \end{definition}

Before proceeding further, we state the following result from \cite{jaydebmansirakshit2022}, which makes it evident why the operators in the constructions work. The last observation is not mentioned in \cite{jaydebmansirakshit2022} explicitly but it follows by a straightforward verification. 

\begin{lemma}\label{1 prop properties of j-th diagonal operator}
Let $\mathcal{H}$ be a Hilbert space. Let $U,\widetilde{U}$ be commuting unitaries on $\mathcal{H}$.Then on $H^{2}_{\mathcal{H}}(\mathbb{D}^{n})$, the following statements hold for $1\le i,j,r\le n$. 
 \begin{enumerate}\renewcommand{\labelenumi}{(\roman{enumi})}
 \item $\left(D_{j}[U]\right)^{*}=D_{j}[U^{*}]$.
 \item $D_{i}[U]D_{j}[\widetilde{U}]=D_{j}[\widetilde{U}]D_{i}[U]$.
 \item $M_{z_{i}}D_{j}[U]=D_{j}[U]M_{z_{i}}$ if $i\neq j$.
 \item $M_{z_{i}}^{*}D_{i}[U]=(I_{{H}^{2}_{\mathcal{H}}(\mathbb{D}^{n})}\otimes U)D_{i}[U]M_{z_{i}}^{*}$.
 \item $D_{i}[U](I_{{H}^{2}(\mathbb{D}^{n})}\otimes T)=(I_{{H}^{2}(\mathbb{D}^{n})}\otimes T)D_{i}[U]$ for every $T\in B(\cl H)$ such that $TU=UT$. 
\end{enumerate}
\end{lemma}

The following result from \cite{jaydebmansirakshit2022} gives a recipe for constructing doubly twisted isometries from a set of unitaries.  

\begin{prop}\label{1 prop example of doubly twisted isometry}
Let $\mathcal{H}$ be a Hilbert space, and let $\{U_{ij}:i,j=1,\dots,n,i\neq j\}$ be a commuting family of unitaries on $\mathcal{H}$ such that $U_{ji}=U_{ij}^{*}$ for all $i\neq j$. Fix $m\in \{1,\dots,n\}$ and consider $(n-m)$ unitary operators $\{U_{m+1},\dots, U_{n}\}$ in $\mathcal{B(H)}$ such that 
$$U_{i}U_{j}=U_{ij}U_{j}U_{i} \quad \text{and} \quad U_{i}U_{pq}=U_{pq}U_{i}$$
for all $m+1\leq i\neq j\leq n,$ and $1\leq p\neq q\leq n$. Then there exists an $n$-tuple of doubly twisted isometry on $H^{2}_{\mathcal{H}}(\mathbb{D}^m)$ with respect to $\{I\otimes U_{ij}\}_{i<j}$ defined as
$$
M_{i} =
\begin{cases}
M_{z_1} & if \ i=1\\
M_{z_{i}}\left(D_{1}[U_{i1}]\dots D_{i-1}[U_{ii-1}]\right) &\text{if } \ 2\leq i\leq m\\
\left(D_{1}[U_{i1}]\dots D_{m}[U_{im}]\right)\left(I_{H^2(\bb D^m)}\otimes U_{i}\right)& \text{if } \ m+1\leq i\leq n.
\end{cases}
$$
Moreover $M_{1},\dots,M_{m}$ are shifts and $M_{m+1},\dots,M_{n}$ are unitaries on $H^2_{\cl H}(\bb D^m)$.
\end{prop}

\begin{remark} We make the following remark concerning the unitaries $U_i, m+1\le i\le n$, appearing in Proposition \ref{1 prop example of doubly twisted isometry}. The condition $U_iU_j=U_{ij}U_jU_i$ in the proposition automatically implies $U_i^*U_j=U_{ij}^*U_jU^{*}_i$, since $U_i$ and $U_j$ are unitaries. 
Moreover, it is precisely the condition $U_i^*U_j=U_{ij}^*U_jU^{*}_i$ for $m+1\le i\ne j\le n$ that ensures $M_i^*M_j=(I_{H^2(\bb D^m)}\otimes U_{ij}^*)M_jM_i^*$ for all $1\le i\ne j \le n$.  

However, the implication
$$
U_iU_j=U_{ij}U_jU_i \ \ \implies \ \ U_i^*U_j=U_{ij}^*U_jU_i^*
$$ 
need not hold if $U_i$ and $U_j$ are merely assumed to be isometries. Nonetheless, if each $U_i$ is an isometry commuting with every $U_{jk}$, then the operators $M_i$ remain isometries and continue to commute with $I_{H^2(\bb D^m)}\otimes U_{jk}$. Consequently, one can verify that the arguments used by the authors of \cite{jaydebmansirakshit2022} to prove Proposition \ref{1 prop example of doubly twisted isometry} actually shows that $(M_1, \dots, M_n)$ is a doubly twisted isometry with respect to the unitaries $I_{H^2(\bb D^m)}\otimes U_{ij}$, even when one starts with isometries  $U_{m+1}, \dots, U_n$ provided one imposes the additional condition $U_{i}^*U_{j}=U_{ij}^*U_{j}U_{i}^*$ in place of $U_iU_j=U_{ij}U_jU_i$ for 
$m+1\le i\ne j\le n$, noting that the latter condition is in any case implied by the former (\cite[Lemma 3.1]{jaydebmansirakshit2022}).  
\end{remark}
 
 The observation in the above remark plays a key role in the construction below of a doubly twisted near-isometry, which constitutes the main result of this section. Its proof follows closely the ideas of Proposition \ref{1 prop example of doubly twisted isometry} and is therefore omitted.

\begin{theorem}\label{doubly-ex1}
 Let $\cl H$ be a Hilbert space and let $\{U_{ij}:1\leq i\ne j\leq n\}$ be a commuting family of unitaries on $\mathcal{H}$ such that $U_{ji}=U_{ij}^{*}$ for 
 all $i\neq j.$ For a fixed $m\in I_{n}$, let $\{T_{m+1},\dots ,T_{n}\}$ be $n-m$ near-isometries on $\cl H$ such that 
 $T_{i}T_{j}=U_{ij}T_{j}T_{i}$, $T_{i}^{*}T_{j}=U_{ij}^{*}T_{j}T_{i}^{*}$, $T_{i}U_{pq}=U_{pq}T_{i}$ for all $m+1\leq i\neq j\leq n$ and $1\leq p\neq q\leq n$.
Then the operators $M_i$ on $H^2_{\cl H}(\bb D^m)$ defined as
 $$M_{i} =
\begin{cases}
M_{z_1} &\text{if } \ i=1\\
M_{z_{i}}\left(D_{1}[U_{i1}]\dots D_{i-1}[U_{ii-1}]\right) &\text{if } \ 2\leq i\leq m\\
\left(D_{1}[U_{i1}]\dots D_{m}[U_{im}]\right)\left(I\otimes T_{i}\right) & \text{if} \ m+1\leq i\leq n
\end{cases}$$
constitute a doubly twisted near-isometry $(M_1, \dots, M_n)$ with respect to unitaries $\{I_{H^2(\bb D^m)}\otimes U_{ij}: 1\le i\ne j\le n\}$. Moreover for $1\leq i\leq m, M_{i}$ is a shift operator, and for $m+1\leq i\leq n, M_{i}$ is an invertible operator if $T_{i}$ is an invertible operator on $\mathcal{H}$.
\end{theorem}

We conclude this section with some simple yet fundamental observations about a doubly twisted near-isometry. 

\begin{lemma}\label{1 le properties of dtni on wa}
 Let $T=(T_1, T_2, \dots, T_n)$ be a doubly twisted near-isometry and let $A\subseteq I_{n}$. Then for $j\in I_{n}\setminus A$, we have
\begin{enumerate}
\renewcommand{\labelenumi}{(\roman{enumi})}
\item $\mathcal{W}_{A}$ reduces $T_{j}$; 
\item $\mathcal{W}_{A}\ominus T_{j}\mathcal{W}_{A}=\mathcal{W}_{A\cup \{j\}}$;
\item $\mathcal{W}_{A}$ reduces $U_{ij}$ and $ U_{ij}\mathcal{W}_{A}=\mathcal{W}_{A}$ for $1\leq i\neq j \leq n$; 
\item $T_{j}$ is a near-isometry on $\mathcal{W}_{A}$;
\end{enumerate} 
where $\cl W_A=\bigcap\limits_{i\in A}ker T_i^*$.   
\end{lemma}

\begin{proof} Suppose $x\in \mathcal{W}_{A}$, then $T_{i}^{*}x=0$ for all $i\in A$. Let $j\in I_{n}\setminus A$. Then $T_{i}^{*}T_{j}x=U_{ij}^{*}T_{j}T_{i}^{*}x=0$ and 
$T_i^*T_j^*x=U_{ij}T_j^*T_i^*x=0$. This shows that $ker(T_i^*)$ reduces $T_j$ for all $i\in A$. Thus, $\mathcal{W}_{A}$ reduces $T_j$ for all $j\in I_{n}\setminus A$. 

To prove $(ii)$, let $x\in \mathcal{W}_{A}\ominus T_{j}\mathcal{W}_{A}$. Then for any $y\in \mathcal{W}_{A}$, we have $0=\langle x,T_{j}y \rangle=\langle T_{j}^{*}x,y\rangle$. Thus $T_{j}^{*}x\in \mathcal{W}_{A}^{\perp}$. Also from part $(i)$ of this lemma, $\mathcal{W}_{A}$ reduces $T_{j}$, which implies 
$T_{j}^{*}x\in \mathcal{W}_{A}$; thus, $T_{j}^{*}x\in \mathcal{W}_{A}\bigcap \mathcal{W}_{A}^{\perp}$. This gives $T_{j}^{*}x=0$, and hence $x\in \mathcal{W}_{A}\bigcap ker T_{j}^{*}=\cl W_{A\cup\{j\}}$. To prove the other containment, let $x\in \mathcal{W}_{A\cup\{j\}}$. Then $x\in \mathcal{W}_{A}$ and $x\in ker T_{j}^{*}$.  This means, $T_j^*x=0$, which implies that $\langle x,T_{j}y \rangle=\langle T_{j}^{*}x,y\rangle=0$ for all $y\in \mathcal{W}_{A}$. Therefore, $x\perp T_{j}\mathcal{W}_{A}$, which yields $x\in \mathcal{W}_{A}\ominus T_{j}\mathcal{W}_A$.

To prove $(iii)$, let $x\in \mathcal{W}_{A}$. Then $T_{k}^{*}x=0$ for all $k\in A$. Hence $T_{k}^{*}U_{ij}x=U_{ij}T_{k}^{*}x=0$ for all $k\in A$. Thus $U_{ij}x \in ker T_{k}^{*}$, which implies $U_{ij}(ker T_{k}^{*})\subseteq kerT_{k}^{*}$ , which further gives $U_{ij}\mathcal{W}_{A}\subseteq \mathcal{W}_{A}$. Similarly $T_{k}^{*}U_{ij}^{*}x=U_{ij}^{*}T_{k}^{*}x=0$ for all $k\in A$. Hence $U_{ij}^{*}x\in ker T_{k}^{*}$, so $U_{ij}^{*}ker T_{k}^{*}\subseteq ker T_{k}^{*}$, which gives $U_{ij}^{*}\mathcal{W}_{A}\subseteq \mathcal{W}_{A}$. Thus, $\mathcal{W}_{A}$ reduces $U_{ij}$. Note that, we also have $\cl W_A\subseteq U_{ij}\cl W_A$, since 
$U_{ij}$ is a unitary. Hence, $U_{ij}\cl W_A=\cl W_A$.

Lastly, $(iv)$ follows from Lemma \ref{1 shift on reducing subspace}, since $\mathcal{W}_{A}$ reduces $T_{j}$.
\end{proof}

\section{A Wold-type decomposition for doubly twisted near-isometries}\label{doubly-twisted}
In this section, we prove that every doubly twisted near-isometry admits a Wold-type decomposition, as defined in Section \ref{char-Wold}. This is established in Theorem \ref{1 th wold decom for dtni by induction}. We begin by listing some technical identities used repeatedly in its proof.

\begin{lemma}\label{1 lemma to show wold for dtni by induction}
Let $T_{1}\in \mathcal{B}(\mathcal{H})$ be a bounded below operator. Suppose $\cl W$ is a closed subspace of $\mathcal{H}$ and $T_{2}\in \mathcal{B(H)}$ is an injective operator. Assume there exists a unitary operator $U$ on $\mathcal{H}$ satisfying the following conditions.
\begin{enumerate}
\renewcommand{\labelenumi}{(\roman{enumi})}
 \item $T_{1}U=UT_{1},T_2U=UT_2$
 \item $T_1T_2=UT_2T_1$
 \item $T_{1}\cl W\subseteq \cl W$
 \item $T_{2}^{k}\cl W\perp T_{2}^{m}\mathcal{H}$ whenever $0\leq k< m$
\item $U(\cl W)=\cl W$.
\end{enumerate}
 
\noindent Then the following holds.
\begin{enumerate}
\renewcommand{\labelenumi}{(\roman{enumi})}
\item $T_{1}^{k_{1}}T_{2}^{k_{2}}=U^{k_{1}k_{2}}T_{2}^{k_{2}}T_{1}^{k_1}$ for all $k_{1},k_{2}\geq 0$.
\item $T_{1}(\bigcap\limits_{k_{2}\geq 0}T_{2}^{k_{2}}\mathcal{R})=\bigcap\limits_{k_{2}\geq 0}T_{1}T_{2}^{k_{2}}\mathcal{R}$ for any subspace 
$\mathcal{R}$ of $\mathcal{H}$.
\item $T_{1}\left(\bigoplus\limits_{k_{2}\geq 0}T_{2}^{k_{2}}\cl W\right)=\bigoplus\limits_{k_{2}\geq 0}T_{1}T_{2}^{k_{2}}\cl W$.
\item $T_{1}\left(\bigoplus\limits_{k_{2}\geq 0}T_{2}^{k_{2}}\cl W\bigoplus  \bigcap\limits_{k_{2}\geq 0}T_{2}^{k_{2}}\mathcal{H}\right)=\bigoplus\limits_{k_{2}\geq 0}T_{1}T_{2}^{k_{2}}\cl W\bigoplus   \bigcap\limits_{k_{2}\geq 0}T_{1}T_{2}^{k_{2}}\mathcal{H}$.
\item $\bigoplus\limits_{k_{2}\geq 0}T_{2}^{k_{2}}\left(\bigcap\limits_{k_{1}\geq 0}T_{1}^{k_{1}}\cl W\right)=\bigcap\limits_{k_{1}\geq 0}T_{1}^{k_{1}}\left(\bigoplus\limits_{k_{2}\geq 0} T_{2}^{k_{2}}\cl W\right)$.
\end{enumerate}
\end{lemma}

We now prove a Wold-type decomposition for a doubly twisted near-isometry. 

\begin{theorem}\label{1 th wold decom for dtni by induction}
Let $T=(T_{1},T_{2},...,T_{n})$ be a doubly twisted near-isometry. Then $T$ admits a Wold-type decomposition such that
$\mathcal{H}=\bigoplus_{A\in I_{n}}\mathcal{H}_{A},$ where for a non-empty $A\subseteq I_{n}$, 
$$\mathcal{H}_{A}=\bigoplus_{\bs k\in\mathbb{N}_{0}^{|A|}}T_{A}^{\bs k}(\bigcap_{\bs \ell\in\mathbb{N}_{0}^{n-|A|}} T_{I_{n}\setminus A}^{\bs \ell}\mathcal{W}_{A})$$
and for $A=\phi$,
$$\mathcal{H}_{A}=\bigcap_{\bs k\in\mathbb{N}_{0}^{I_{n}}}T_{I_{n}}^{\bs k}\mathcal{H}.$$
\end{theorem} 

\begin{proof}
We established the Wold-type decomposition by induction on $m$ operator. We first prove for the case $m=2$. 
By the Wold-type decomposition for a near-isometry $T_{1}$, we obtain
 \begin{equation}\label{eq1 wold}
\mathcal{H}=\bigoplus_{k_{1}\geq 0}T_{1}^{k_{1}}\mathcal{W}_{1}\bigoplus \bigcap_{k_{1}\geq 0}T_{1}^{k_{1}}\mathcal{H},
\end{equation}
where $\cl W_1=kerT_1^*$. Using Lemma \ref{1 le properties of dtni on wa}, $T_{2}$ is a near-isometry on $\cl W_1$; therefore,   

\begin{equation}\label{wold2}
 \mathcal{W}_{1}=\bigoplus_{k_{2}\geq0}T_{2}^{k_{2}}(\mathcal{W}_{1}\ominus T_{2}\mathcal{W}_{1})\bigoplus\bigcap_{k_{2}\geq0}T_{2}^{k_{2}}\mathcal{W}_{1}.
\end{equation}

Substituting (\ref{wold2}) into (\ref{eq1 wold}), we obtain

\begin{equation*}
\mathcal{H}=\bigoplus_{k_{1}\geq 0}T_{1}^{k_{1}}\left(\bigoplus_{k_{2}\geq 0}T_{2}^{k_{2}}\mathcal{W}_{\{1,2\}}\bigoplus \bigcap_{k_{2}\geq 0}T_{2}^{k_{2}}\mathcal{W}_{1}\right)\bigoplus\bigcap_{k_{1}\geq0}T_{1}^{k_{1}}\mathcal{H}.
\end{equation*}

Then, using Lemma \ref{1 lemma to show wold for dtni by induction}, we obtain

\begin{equation}\label{eq3 wold}
\mathcal{H}=\bigoplus_{k_{1},k_{2}\geq 0}T_{1}^{k_{1}}T_{2}^{k_{2}}\mathcal{W}_{\{1,2\}}\bigoplus\bigoplus_{k_{1}\geq 0}T_{1}^{k_{1}}\left(\bigcap_{k_{2}\geq 0}T_{2}^{k_{2}}\mathcal{W}_{1}\right)\bigoplus \bigcap_{k_{1}\geq 0}T_{1}^{k_{1}}\mathcal{H}.
\end{equation}

Additionally, by Wold-type decomposition for the near-isometry $T_{2}$ , we have
\begin{equation*}
  \mathcal{H}=\bigoplus_{k_{2}\geq0}T_{2}^{k_{2}}\mathcal{W}_{2}\bigoplus \bigcap_{k_{2}\geq 0}T_{2}^{k_{2}}\mathcal{H}.  
\end{equation*}
Then, for $k_{1}\geq 0$, we have
\begin{eqnarray*}
 T_{1}^{k_{1}}\mathcal{H} &=& T_{1}^{k_{1}}\left(\bigoplus_{k_{2}\geq 0}T_{2}^{k_{2}}\mathcal{W}_{2}\bigoplus \bigcap_{k_{2}\geq 0}T_{2}^{k_{2}}\mathcal{H}\right)\\
&=&\bigoplus_{k_{2}\geq 0}T_{1}^{k_{1}}T_{2}^{k_{2}}\mathcal{W}_{2}\bigoplus \bigcap_{k_{2}\geq 0}T_{1}^{k_{1}}T_{2}^{k_{2}}\mathcal{H} 
\quad ({\rm using \ Lemma \ \ref{1 lemma to show wold for dtni by induction} (iv)})\\   
&=& \bigoplus_{k_{2}\geq 0}T_{2}^{k_{2}}T_{1}^{k_{1}}\mathcal{W}_{2}\bigoplus \bigcap_{k_{2}\geq 0}T_{1}^{k_{1}}T_{2}^{k_{2}}\mathcal{H}, 
\end{eqnarray*}
since $T_{1}^{k_{1}}T_{2}^{k_{2}}=U_{12}^{k_{1}k_2}T_{2}^{k_{2}}T_{1}^{k_{1}}$ and $U_{12}\cl W=\cl W$. Now, taking the intersection over all $k_{1}\geq 0$ and using Lemma  \ref{1 lemma to show wold for dtni by induction}, we get 
\begin{equation}\label{eq8 wold}
\bigcap _{k_{1}\geq 0}T_{1}^{k_{1}}\mathcal{H}= \bigoplus_{k_{2}\geq 0}T_{2}^{k_{2}}\left(\bigcap _{k_{1}\geq 0}T_{1}^{k_{1}}\mathcal{W}_{2}\right)\bigoplus \bigcap_{k_{1},k_{2}\geq 0}T_{1}^{k_{1}}T_{2}^{k_{2}}\mathcal{H}.     
\end{equation}

Then, using (\ref{eq8 wold}) in (\ref{eq3 wold}), we have 
\begin{eqnarray*}
\mathcal{H} &=& \bigoplus_{k_{1},k_{2}\geq 0} T_{1}^{k_{1}} T_{2}^{k_{2}} \mathcal{W}_{\{1,2\}} \bigoplus\ \bigoplus_{k_{1}\geq 0} T_{1}^{k_{1}}
\left( \bigcap_{k_2\geq 0} T_{2}^{k_{2}} \mathcal{W}_{1} \right) \\
& & \bigoplus \bigoplus_{k_{2}\geq 0} T_{2}^{k_{2}}\left( \bigcap_{k_{1}\geq 0} T_{1}^{k_{1}} \mathcal{W}_{2} \right)
   \bigoplus \bigcap_{k_{1},k_{2}\geq 0} T_{1}^{k_{1}} T_{2}^{k_{2}} \mathcal{H}.
\end{eqnarray*}

This establishes the Wold-type decomposition for $m=2$ case. 

Suppose $m<n$, and $(T_1,T_2,\dots, T_{m})$ admits Wold-type decomposition. Then 
\begin{equation}\label{eq10 wold}
 \mathcal{H}=\bigoplus_{A\subseteq I_{m}}\,\mathcal{H}_{A}, \ \ \ {\rm where} \  \mathcal{H}_{A}=
 \bigoplus_{\bs k\in \mathbb{N}_{0}^{|A|}}T_{A}^{\bs k}\left(\bigcap_{\bs \ell\in \mathbb{N}_{0}^{m-|A|}}T_{I_{m}\setminus A}^{\bs \ell}\mathcal{W}_{A}\right).
\end{equation}

We need to show that $(T_{1},T_{2},\dots, T_{m},T_{m+1})$ admits Wold-type decomposition. For $A\subseteq I_{m}$, $\mathcal{W}_{A}$ reduces $T_{m+1}$. By Wold-type decomposition on $\mathcal{W}_{A}$, we obtain
\begin{equation*}
\mathcal{W}_{A}=\bigoplus_{k_{m+1}\geq 0}T_{m+1}^{k_{m+1}} \mathcal{W}_{A\cup\{m+1\}}\bigoplus\bigcap_{k_{m+1}\geq 0}T_{m+1}^{k_{m+1}}\mathcal{W}_{A}.
\end{equation*}
Substituting the expression for $\mathcal{W}_{A}$ in (\ref{eq10 wold}), we obtain

\begin{equation*}
\mathcal{H}_{A}=\bigoplus_{\bs k\in \mathbb{N}_{0}^{|A|}}T_{A}^{\bs k}\left(\bigcap_{\bs \ell\in \mathbb{N}_{0}^{m-|A|}}T_{I_{m}\setminus A}^{\bs \ell} \left( \bigoplus_{k_{m+1}\geq 0}T_{m+1}^{k_{m+1}} \mathcal{W}_{A\cup\{m+1\}}\bigoplus\bigcap_{k_{m+1}\geq 0}T_{m+1}^{k_{m+1}}\mathcal{W}_{A}\right)\right).
\end{equation*}

\noindent Using Lemma \ref{1 lemma to show wold for dtni by induction} $(iii)$, we obtain
\begin{eqnarray*}
\mathcal{H}_{A} &=& \bigoplus_{\bs k\in \mathbb{N}_{0}^{|A|}}T_{A}^{\bs k}\left(\bigcap_{\bs \ell\in \mathbb{N}_{0}^{m-|A|}}T_{I_{m}\setminus A}^{\bs \ell}\left( \bigoplus_{k_{m+1}\geq 0}T_{m+1}^{k_{m+1}} \mathcal{W}_{A\cup\{m+1\}}\right)\right. \\
&& \hspace{2 cm}  \left. \bigoplus \bigcap_{\substack{\bs \ell\in \mathbb{N}_{0}^{m-|A|},\\ k_{m+1}\geq 0}} T_{I_{m}\setminus A}^{\bs \ell}T_{m+1}^{k_{m+1}}\mathcal{W}_{A}\right).  
\end{eqnarray*}

Hence, using Lemma \ref{1 lemma to show wold for dtni by induction}, we have 
\begin{eqnarray*}
\cl{H}_{A}
 = && 
\bigoplus_{\substack{\bs k\in\mathbb{N}_{0}^{|A|} \\ k_{m+1}\ge 0}}
T_{A}^{\bs k}\, T_{m+1}^{k_{m+1}}
\left(
    \bigcap_{\bs \ell\in \mathbb{N}_{0}^{\,m-|A|}}
    T_{I_{m}\setminus A}^{\bs \ell}
    \mathcal{W}_{A\cup\{m+1\}}
\right)\\
&& \bigoplus \bigoplus_{\bs k\in \mathbb{N}_{0}^{|A|}}
T_{A}^{\bs k}
\left(
    \bigcap_{\substack{\bs \ell\in \mathbb{N}_{0}^{\,m-|A|} \\ k_{m+1}\ge 0}}
    T_{I_{m}\setminus A}^{\bs \ell}
    T_{m+1}^{k_{m+1}}
    \mathcal{W}_{A}
\right),
\end{eqnarray*}
since $T_{m+1}^{k_{m+1}}T_{I_{m}\setminus A}^{\bs \ell}\mathcal{W}_{A\cup\{m+1\}}=T_{I_{m}\setminus A}^{\bs \ell}T_{m+1}^{k_{m+1}}\mathcal{W}_{A\cup\{m+1\}}$ 
and $U_{ij}(\mathcal{W}_{A\cup\{m+1\}})=\mathcal{W}_{A\cup\{m+1\}}$ for all $1\leq i,j\leq m$.

This proves $\mathcal{H}=\bigoplus_{B\subseteq I_{m}\cup\{m+1\}}\mathcal{H}_{B}$. This claim the induction hypothesis, and hence doubly twisted near-isometry $T$ admit Wold-type decomposition.

We now show for $\mathcal{H}_A$ reduces $T_{j}$ for all $A\subseteq I_n$ and $j\in I_{n}$.
Suppose $q_j\in I_n\setminus A$ and $x\in \cl H_A$. Then for all $\bs \ell\in \mathbb{N}_{0}^{n-|A|}$, there exists $x_{\bs k, \bs\ell}\in \mathcal{W}_{A}$ such that 
$$
x=\sum_{\bs k\in \mathbb{N}_{0}^{|A|}}T_{A}^{\bs k}T_{I_{n}\setminus A}^{\bs \ell}x_{\bs k,\bs \ell}.
$$ 
Then $T_{q_j}x=\sum\limits_{\bs k\in \mathbb{N}_{0}^{|A|}}T_{A}^{\bs k}T_{I_{n}\setminus A}^{\bs \ell}P(U)T_{q_j}x_{\bs k, \bs \ell}$, where 
$$P(U)=U_{q_{j}a_{1}}^{k_{1}}\cdots U_{q_{j}a_{m}}^{k_{m}}U_{q_jq_1}^{\ell_1}\cdots U_{q_{j}q_{j-1}}^{\ell_{j-1}}U_{q_jq_{j+1}}^{l_{j+1}}\cdots U_{q_{j}q_{n-m}}^{\ell_{n-m}}.$$ 
Since $\mathcal{W}_{A}$ reduces $T_{q_j}$ and $U_{rs}$ for all $1\leq r\ne s\leq n$, we have $T_{q_j}\mathcal{H}_{A}\subseteq \mathcal{H}_{A}$. Thus, $\cl H_A$ is invariant under $T_j$ for all $j\in I_n\setminus A$. Similarly, $\cl H_A$ is invariant under every $T_i, \ i\in A.$ So, $\mathcal{H}_{A}$ reduces $T_{j}$ for all $A\subseteq I_{n}$ and $j\in I_{n}$.  

Finally, by definition, $\mathcal{H}_{A}\subseteq \mathcal{H}_{T_{i},S}$ for every $i\in A$, and  
\begin{equation*}
\mathcal{H}_{A}\subseteq \bigoplus_{\bs k\in\mathbb{N}_{0}^{|A|}}T_{A}^{\bs k}\left(\bigcap_{\bs \ell\in\mathbb{N}_{0}^{n-|A|}} T_{I_{n}\setminus A}^{\bs \ell}\mathcal{H}\right) \subseteq 
\bigoplus_{\bs k\in\mathbb{N}_{0}^{|A|}}T_{A}^{\bs k}\left(\bigcap_{m\ge 0}T_{j}^{m}\mathcal{H}\right)\subseteq \mathcal{H}_{T_j,I},   
\end{equation*}
for every $j\in I_n\setminus A$. Hence, $T_i$ is a shift on $\cl H_A$ whenever $i\in A$ and an invertible operator whenever $i\in I_n\setminus A.$
\end{proof}

\begin{remark} Theorem \ref{1 th wold decom for dtni by induction} shows that a Wold-type decomposition exists for a doubly twisted near-isometry and provides a representation of each summand. Moreover, as noted in Remark \ref{1 remark unique rep of Wold for n-tuple of near-isometry}, such a decomposition, if it exists for a tuple of near-isometries, must be unique, and the precise form of the summands is also identified there. Hence, the two representations of the summands must coincide; that is,  
$$
\mathcal{H}_{A}= \bigoplus_{\bs k\in \mathbb{N}_{0}^{|A|}}T_{A}^{\bs k}\left(\bigcap_{\bs\ell\in \mathbb{N}_{0}^{m-|A|}}T_{I_{m}\setminus A}^{\bs \ell}\mathcal{W}_{A}\right)
 =\left(\bigcap_{i\in I_{n}\setminus A}\mathcal{H}_{T_{i},I}\right)\bigcap \left(\bigcap_{k\in A}\mathcal{H}_{T_{k},S}\right).
$$
\end{remark}

\begin{remark}\label{cor-isometry}
Since every isometry is, in particular, a near-isometry, the Wold decomposition for a doubly twisted isometry arises as a special case of the Wold-type decomposition established here for a doubly twisted near-isometry. In this way, our result extends the corresponding decomposition obtained in \cite{jaydebmansirakshit2022}. Furthermore, when $U_{ij}=I$ for all $1\leq i\neq j\leq n$, our theorem specializes to the Wold-type decomposition for doubly commuting near-isometries obtained by the first and third authors of the present paper with Pokhriyal in \cite{lata2022multivariable}.
\end{remark}

\section{An alternative proof of a wold-type decomposition for doubly twisted near-isometries}\label{alternate}
In this section, we shall give another poof for the existence of a Wold-type decomposition for a doubly twisted near-isometry. Unlike the proof presented in the last section that was based on induction, here the proof relies on arguments involving orthogonal projections. We shall refer to orthogonal projections simply as projections.
 
\begin{definition}
 A net $\{T_{\alpha}\}_{\alpha}$ converges in the strong operator topology (SOT) to $T$ if $T_{\alpha} x \to Tx$ for all $x\in\mathcal{H}.$ In this case, we write  $T_{\alpha}\xrightarrow{SOT} T$ or $T=SOT-\lim\limits_{\alpha}T_\alpha$. 
\end{definition}

Before proceeding further, we recall the following well-known fact about increasing nets of self-adjoint operators. 

\begin{lemma} \label{1 le SOT of increasing projection}
Let $\{A_\alpha\}_{\alpha\in \Lambda}$ be an increasing net of self-adjoint operators on a Hilbert space $\cl H$, bounded above by a self-adjoint operator on 
$\cl H$. Then $\{A_\alpha\}_{\alpha\in \Lambda}$ converges in SOT to its least upper bound. Moreover, if $A_\alpha$ is a projections, then the SOT limit is the 
projection onto the closure $\bigcup_{\alpha}Ran(A_\alpha)$.
\end{lemma}

The above lemma will be used in the following result to identify the projections onto the shift and the invertible summands $\mathcal{H}_{T,S}$ and $\mathcal{H}_{T,I}$ appearing in the Wold-type decomposition of a near-isometry.

\begin{lemma}\label{1 le projection onto shift and invertible part}
Let $T\in \mathcal{B(H)}$ be a near-isometry with the Wold-type decomposition
$\mathcal{H}=\mathcal{H}_{T,S}\bigoplus \mathcal{H}_{T,I}.$ Then
\begin{enumerate}
\renewcommand{\labelenumi}{(\roman{enumi})}
 \item $T^{n}(T^{n})^{\#}-T^{n+1}(T^{n+1})^{\#}$ is the projection onto $T^{n}(ker T^{*})$.
\item $P_{\mathcal{H}_{T,S}}=SOT-\!\lim\limits_{n \to \infty} \sum_{k=0}^{n} T^{k}(T^{k})^{\#}-T^{k+1}(T^{k+1})^{\#}.$
\item $P_{\mathcal{H}_{T,I}}=SOT-\!\lim\limits_{n \to \infty}  T^{n}(T^{n})^{\#}.$
\end{enumerate}
\end{lemma}

\begin{proof} 
For $n\geq 0$, $T^n\in B(\cl H)$ is bounded below and $(T^n)^{\#}=\left(T^{n*}T^n\right)^{-1}T^{n*}$ is a left inverse of $T^n$. With straightforward arguments, one can easily show that $T^{n}(T^{n})^{\#}$  is the projection onto $T^{n}(\mathcal{H})$.    

Consequently, for each fixed $n$ and each $0\le k\le n$, we have $T^{k}(T^{k})^{\#} -T^{k+1}(T^{k+1})^{\#}$ is the projection onto 
$T^k(\cl H)\ominus T^{k+1}(\cl H)=T^k(kerT^*)$. This proves $(i)$.  

Next, using $(i)$, we have that    
$$
A_{n}:=\sum_{k=0}^{n}T^{k}(T^{k})^{\#}-T^{k+1}(T^{k+1})^{\#}
$$
is the orthogonal projection onto $\bigoplus_{k=0}^{n}T^{k}(ker T^{*}).$ Furthermore, $\{A_n\}$ is an increasing sequence; therefore, 
SOT-$\lim\limits_{n\to \infty}A_n$ exits and is the projection onto 
$$
\overline{\bigcup\limits_{n=0}^\infty\left(\bigoplus_{k=0}^n T^{k}(ker T^{*})\right)}=\bigoplus_{k\ge 0} T^{k}(ker T^{*})=\cl H_{T,S}.
$$ 

Therefore, $\cl H_{T,S}=SOT-\!\lim\limits_{n \to \infty} \sum_{k=0}^{n} \left(T^{k}(T^{k})^{\#}-T^{k+1}(T^{k+1})^{\#}\right),$ which proves $(ii)$.  

Lastly, since $T$ is a near-isometry on $\mathcal{H}$, we can decompose $\cl H$ as  
$$
\mathcal{H}=\bigoplus\limits_{k=0}^nT^k(ker T^{*})\bigoplus T^{n+1}\mathcal{H},
$$ 
for each $n\ge 0$. From this, it is evident that $T^{n+1}(T^{n+1})^{\#}$, the projection onto $T^{n+1}\cl H$, equals $I-A_n$, the projection onto 
$\cl H\ominus \bigoplus\limits_{k=0}^nT^k (ker T^{*})$. Hence,
$$SOT-\!\lim\limits_{n \to \infty} T^{n}(T^{n})^{\#}=I-P_{\cl H_{T,S}}=\cl H_{T,I}.$$ This completes the proof.  
\end{proof}

\begin{lemma}\label{1 le power and hatsh of T_{i} and T_{j}}
Let $T=(T_{1},T_{2},\dots ,T_{n})$ be a doubly twisted near-isometry on $\mathcal{H}$. Then for any $k, \ell\in \mathbb{N}_{0}$ and $1\le i\ne j\le n$, we have 
$$T_{i}^{k}(T_{i}^{k})^{\#}(T_{j}^{\ell})(T_{j}^{\ell})^{\#}=(T_{j}^{\ell})(T_{j}^{\ell})^{\#}T_{i}^{k}(T_{i}^{k})^{\#}.$$
\end{lemma}

\begin{proof} Observe that $T_{i}^{*k}T_{i}^{k}T_{j}^{\ell}=U_{ij}^{k\ell}T_{i}^{*k}T_{j}^{\ell}T_{i}^{k}=U_{ij}^{k\ell}U_{ij}^{*k\ell}T_{j}^{\ell}T_{i}^{*k}T_{i}^{k}=T_{j}^{\ell}T_{i}^{*k}T_{i}^{k}.$ Thus, 
\begin{equation}\label{eq1 hatsh}
   T_{j}^{\ell}(T_{i}^{*k}T_{i}^{k})^{-1}=(T_{i}^{*k}T_{i}^{k})^{-1}T_{j}^{\ell}.
\end{equation}

Similarly,
\begin{equation}\label{eq2 hatsh}
T_{i}^{k}(T_{j}^{*\ell}T_j^{\ell})^{-1}=(T_{j}^{*\ell}T_{j}^{\ell})^{-1}T_{i}^k.  
\end{equation}

Then, 
$(T_i^k)^{\#}T_j^{\ell} =   (T_{i}^{*k}T_{i}^{k})^{-1}T_i^{*k}T_j^{\ell}= U_{ij}^{*k\ell}   (T_{i}^{*k}T_{i}^{k})^{-1}T_j^{\ell}T_i^{*k}=U_{ij}^{*k\ell} T_j^{\ell} (T_i^k)^{\#},$ 
using (\ref{eq1 hatsh}).
Thus, 
\begin{equation}\label{hatsh-1}
(T_{i}^{k})^{\#}T_{j}^{\ell}=U_{ij}^{*k\ell} T_{j}^{\ell}(T_{i}^{k})^{\#}.
\end{equation}

Similarly, 
\begin{equation}\label{hatsh-2}
(T_{i}^{k})(T_{j}^{\ell})^{\#}=U_{ij}^{*k\ell} (T_{j}^{\ell})^{\#}T_{i}^{k}.
\end{equation}
Then,  
\begin{eqnarray}\label{hatsh-3}
(T_{i}^{k})^{\#}(T_{j}^{\ell})^{\#}&=&(T_{i}^{*k}T_{i}^{k})^{-1}T_{i}^{*k}(T_{j}^{*\ell}T_{j}^{\ell})^{-1}T_{j}^{*\ell}\nonumber\\
&=&(T_{i}^{*k}T_{i}^{k})^{-1}(T_{j}^{*\ell}T_{j}^{\ell})^{-1}T_{i}^{*k}T_{j}^{*\ell}\nonumber\\
&=& U_{ij}^{k\ell} (T_{i}^{*k}T_{i}^{k})^{-1}(T_{j}^{*\ell}T_{j}^{\ell})^{-1}T_{j}^{*\ell}T_{i}^{*k}\nonumber\\
&=& U_{ij}^{k\ell} (T_{j}^{*\ell}T_{j}^{\ell})^{-1}T_{j}^{*\ell}(T_{i}^{*k}T_{i}^{k})^{-1}T_{i}^{*k}\nonumber\\
&=& U_{ij}^{k\ell}(T_{j}^{\ell})^{\#} (T_{i}^{k})^{\#}.
\end{eqnarray}

Finally, 
\begin{eqnarray*}
T_{i}^{k}(T_{i}^{k})^{\#}T_{j}^{\ell}(T_{j}^{\ell})^{\#} &=& U_{ij}^{*kl}T_{i}^{k}T_{j}^{\ell}(T_{i}^{k})^{\#}(T_{j}^{\ell})^{\#} \ \ {\rm using \ \ (\ref{hatsh-1}})\\
&=& T_j^{\ell} T_{i}^{k}(T_i^k)^{\#}(T_{j}^{\ell})^{\#}\\
&= & U_{ij}^{k \ell} T_j^{\ell}T_{i}^{k}(T_{j}^{\ell})^{\#}(T_{i}^{k})^{\#} \ \ {\rm using \ \ (\ref{hatsh-3}})\\
&=& T_{j}^{\ell}(T_{j}^{\ell})^{\#}T_i^k(T_i^k)^{\#} \ \ {\rm using \ \ (\ref{hatsh-2}}).\\
\end{eqnarray*}
This completes the proof. 
\end{proof}

\begin{prop}\label{1 prop projection commutes in dtni}
Let $T=(T_{1},T_{2},...,T_{n})$ be an n-tuple of doubly twisted near-isometry. Then for $1\le i\ne j\le n$, we have
\begin{enumerate}
\renewcommand{\labelenumi}{(\roman{enumi})}
\item $T_{i}^{k}(T_{i}^{k})^{\#}P_{\mathcal{H}_{T_{j},S}}=P_{\mathcal{H}_{T_{j},S}}T_{i}^{k}(T_{i}^{k})^{\#}$ for all $k\ge 0$; 
\item $P_{\mathcal{H}_{T_{i},S}}P_{\mathcal{H}_{T_{j},S}}=P_{\mathcal{H}_{T_{j},S}}P_{\mathcal{H}_{T_{i},S}};$
\item $P_{\mathcal{H}_{T_{i},S}}P_{\mathcal{H}_{T_{i},I}}=P_{\mathcal{H}_{T_{i},I}}P_{\mathcal{H}_{T_{i},S}}=0$.
\end{enumerate}
\end{prop}

\begin{proof} Using Lemma (\ref{1 le power and hatsh of T_{i} and T_{j}}), we have 
$$T_{i}^{k}(T_{i}^{k})^{\#} T_{j}^{\ell}(T_{j}^{\ell})^{\#}=T_{j}^{\ell}(T_{j}^{\ell})^{\#}T_{i}^{k}(T_{i}^{k})^{\#}
$$
for every $\ell\ge 0.$ Thus, $T_{i}^{k}(T_{i}^{k})^{\#}$ commutes with $\sum\limits_{l=0}^{n}\left(T_{j}^{\ell}(T_{j}^{\ell})^{\#}-T_{j}^{\ell+1}(T_{j}^{\ell+1})^{\#}\right)$ for every $n$. Therefore, $T_i^k(T_i^k)^{\#}$ commutes with $P_{\mathcal{H}_{T_{j},S}}$, since  
$$
P_{\mathcal{H}_{T_{j},S}}=SOT-\!\lim\limits_{n \to \infty} \sum_{l=0}^{n}T_{j}^{\ell}(T_{j}^{\ell})^{\#}-T_{j}^{\ell+1}(T_{j}^{\ell+1})^{\#},
$$
as proved in Lemma \ref{1 le projection onto shift and invertible part}. This settles $(i)$. The proof of $(ii)$ follow from $(i)$ and Lemma \ref{1 le projection onto shift and invertible part}, and $(iii)$ follows immediately from the fact that $P_{\mathcal{H}_{T_{i},S}}$ and $P_{\mathcal{H}_{T_{i},I}}$ are orthogonal to each other.
\end{proof}

\begin{lemma}\label{1 le intersection of T_{i}W_{i}}
Let $T=(T_{1},\dots,T_{n})$ be an n-tuple of doubly twisted near-isometry. Then for ${\bs k}=(k_{1}, k_{2},\dots,k_{n})\in \mathbb{N}^{n}_{0}$, we have
$$
\bigcap_{1\leq i\leq n}T_{i}^{k_{i}}(ker T_{i}^{*})=T_{1}^{k_{1}}\dots T_{n}^{k_{n}}\left(\bigcap_{1\leq i\leq n}ker T_{i}^{*}\right).
$$
\end{lemma}

\begin{proof} The containment 
$$
T_{1}^{k_{1}}\dots T_{n}^{k_{n}}\left(\bigcap\limits_{1\leq i\leq n}ker T_{i}^{*}\right)\subseteq \bigcap\limits_{1\leq i\leq n}T_{i}^{k_{i}}(ker T_{i}^{*})
$$ 
follows simply by using $T_i^{k_i}T_j^{k_j}=U_{ij}^{k_ik_j}T_j^{k_j}T_i^{k_i}$ and $kerT_i^*$ is invariant under $T_j$ whenever $i\ne j.$
We shall prove the reverse containment using induction. We first prove the containment for $n=2$, that is,  
\begin{equation}\label{eq1 int of dtni}
T_{1}^{k_{1}}(ker T_{1}^{*})\bigcap T_{2}^{k_{2}}(ker T_{2}^{*})\subseteq T_{1}^{k_{1}}T_{2}^{k_{2}}(ker T_{1}^{*}\bigcap ker T_{2}^{*}).
\end{equation}
Suppose $x\in T_{1}^{k_{1}}(ker T_{1}^{*})\bigcap T_{2}^{k_{2}}(ker T_{2}^{*}) $. Then there exist $x_{1}\in ker T_{1}^{*}$ and $x_{2}\in ker T_{2}^{*}$ such that 

$$
x=T_{1}^{k_{1}}x_{1}=T_{2}^{k_{2}}x_{2}.
$$
This implies $x_{1}=(T_{1}^{k_{1}})^{\#}T_{2}^{k_{2}}x_{2}=T_2^{k_2}(T_{1}^{k_{1}})^{\#}U_{12}^{*k_1k_2}x_2$. But, $kerT_2^*$ reduces $U_{12}$ and $T_1$. Therefore, $x_1\in T_2^{k_2}(kerT_2^*)$, which shows that $x\in T_1^{k_1}T_2^{k_2}(kerT_2^*)$.  

Similarly, $x\in T_{1}^{k_{1}}T_{2}^{k_{2}}(ker T_{1}^{*}).$ As $T_{1}^{k_{1}}T_{2}^{k_{2}}$ is one-one, we deduce that $x\in T_{1}^{k_{1}}T_{2}^{k_{2}}(ker T_{1}^{*}\cap ker T_{2}^{*})$, which establishes (\ref{eq1 int of dtni}).

Suppose induction hypothesis holds for $m-1$ , that is,
\begin{equation}\label{eq3 int of dtni}
\bigcap_{1\leq i\leq m-1}T_{i}^{k_{i}}(ker T_{i}^{*})\subseteq T_{1}^{k_{1}}\dots T_{m-1}^{k_{m-1}}\left(\bigcap_{1\leq i\leq m-1}ker T_{i}^{*}\right).
\end{equation}
We need to show 
$$
\bigcap_{1\leq i\leq m}T_{i}^{k_{i}}(ker T_{i}^{*})\subseteq T_{1}^{k_{1}}\dots T_{m}^{k_{m}}\left(\bigcap_{1\leq i\leq m}ker T_{i}^{*}\right).
$$  

Let $x\in \bigcap\limits_{1\leq i\leq m}T_{i}^{k_{i}}(ker T_{i}^{*})$. Then 
$x\in \bigcap\limits_{1\leq i\leq m-1}T_{i}^{k_{i}}(ker T_{i}^{*})\bigcap T_{m}^{k_{m}}ker T_{m}^{*},$ which, using (\ref{eq3 int of dtni}), implies that there exist 
$y_{1}\in\bigcap\limits_{1\le i\le m-1}kerT_i^*$ and $y_{2}\in kerT_m^*$ such that 
\begin{equation}\label{eq5 int of dtni}
x=T_{1}^{k_{1}}\dots T_{m-1}^{k_{m-1}}y_{1}=T_{m}^{k_{m}}y_{2}.
\end{equation}

Then, we can write 
$$
y_{1}=(T_{m-1}^{k_{m-1}})^{\#}\dots(T_{1}^{k_{1}})^{\#}T_{m}^{k_{m}}y_{2}=T_m^{k_m}(T_{m-1}^{k_{m-1}})^{\#}\dots(T_{1}^{k_{1}})^{\#}Vy_2,
$$
where $V=U_{(m-1)m}^{*k_{m-1}k_m}\cdots U_{1m}^{*k_{1}k_m}.$ Since, $kerT_m^*$ reduces $T_1, \dots, T_{m-1}$ and $U_{1m}, \dots, U_{(m-1)m}$, therefore $y_1\in T_m^{k_m}(ker T_m^*)$. This implies that $x\in T_1^{k_1}\dots T_m^{k_m}(kerT_m^*)$.  

On the other hand, $T_{1}^{k_{1}}\cdots T_{m-1}^{k_{m-1}}y_{1}=T_{m}^{k_{m}}y_{2}$ implies that 
$$
y_2=(T_m^{k_m})^{\#}T_{1}^{k_{1}}\cdots T_{m-1}^{k_{m-1}}y_{1}=T_{1}^{k_{1}}\dots T_{m-1}^{k_{m-1}}(T_m^{k_m})^{\#}Ry_1,
$$
where $R=U_{m1}^{*k_mk_1}\cdots U_{m(m-1)}^{*k_mk_{m-1}}$. Again, since $\bigcap\limits_{1\leq i\leq m-1}ker T_{i}^{*}$ reduces $T_m$ and $U_{m(m-1)}, \dots, U_{m1}$, we conclude that $y_2\in T_1^{k_1}\cdots T_{m-1}^{k_{m-1}}\left(\bigcap\limits_{1\le i\le m-1}kerT_i^*\right)$. Thus, using (\ref{eq5 int of dtni}) and the facts that $T_mT_i=U_{mi}T_iT_m$ 
and $kerT_m^*$ reduces $U_{mi}$ for all $1\le i\le m-1$, we obtain that $x\in T_1^{k_1}\cdots T_m^{k_m} \left(\bigcap\limits_{1\le i\le m-1}kerT_i^*\right).$  

Lastly, as $T_{1}^{k_{1}}\dots T_{m-1}^{k_{m-1}}T_{m}^{k_{m}}$ is injective, we obtain
\begin{equation*}
x\in  T_{1}^{k_{1}}\dots T_{m-1}^{k_{m-1}}T_{m}^{k_{m}}\left(\bigcap_{1\leq i\leq m-1}ker T_{i}^{*}\bigcap ker T_{m}^{*}\right). 
\end{equation*}
Hence, by induction, the reverse containment holds true for a general $n$. This completes the proof.
\end{proof}

\begin{lemma}\label{1 le tj and hatsh commutes with tak}
Let $T=(T_{1},\dots, T_n)$ be doubly twisted near-isometry. For $A\subseteq I_{n}$ and for $j\in I_{n}\setminus A$, we have $T_{j}^{\ell_{j}}(T_{j}^{\ell_{j}})^{\#}T_{A}^{\bs k}= T_{A}^{\bs k}T_{j}^{\ell_{j}}(T_{j}^{\ell_{j}})^{\#}$ for all $\bs k\in \bb N_0^{|A|}$ and $\ell_j\ge 0$.   
\end{lemma}

\begin{proof} Let $|A|=m$. Then, for $\bs k=(k_1, \dots, k_m)\in \bb N_0^{m}$ and $\ell_j\ge 0$, we have 
\begin{equation}\label{1 le tj and hatsh commutes with tak a}
T_{j}^{\ell_{j}}T_{A}^{\bs k}=T_{j}^{\ell_{j}}T_{a_{1}}^{k_{1}}\cdots T_{a_{m}}^{k_{m}}=T_{A}^{\bs k}T_{j}^{\ell_{j}}U_{ja_{1}}^{\ell_{j}k_{1}}\cdots U_{ja_{m}}^{\ell_{j}k_{m}}.    
\end{equation}
Similarly 
\begin{equation}\label{1 le tj and hatsh commutes with tak b}
T_{j}^{*\ell_{j}}T_{A}^{\bs k}=T_{A}^{\bs k}T_{j}^{*\ell_{j}}U_{ja_{1}}^{*\ell_{j}k_{1}}\cdots U_{ja_{m}}^{*\ell_{j}k_{m}}.
\end{equation}
Since $\{U_{ij}\}_{1\leq i\neq j\leq n}$ are commuting families of unitaries, then using (\ref{1 le tj and hatsh commutes with tak a}) and (\ref{1 le tj and hatsh commutes with tak b}), we obtain $$
T_{j}^{*\ell_{j}}T_{j}^{\ell_{j}}T_{A}^{\bs k}=T_{A}^{\bs k}T_{j}^{*\ell_{j}}T_{j}^{\ell_{j}}.
$$ 
Hence 
\begin{equation}\label{1 le tj and hatsh commutes with tak c}
(T_{j}^{*\ell_{j}}T_{j}^{\ell_{j}})^{-1}T_{A}^{\bs k}=T_{A}^{\bs k}(T_{j}^{*\ell_{j}}T_{j}^{\ell_{j}})^{-1}.
\end{equation}
From (\ref{1 le tj and hatsh commutes with tak b}) and (\ref{1 le tj and hatsh commutes with tak c}), we obtain
\begin{equation}\label{1 le tj and hatsh commutes with tak d}
(T_{j}^{\ell_{j}})^{\#}T_{A}^{\bs k}=(T_{j}^{*\ell_j}T_{j}^{\ell_{j}})^{-1}T_{j}^{*\ell_{j}}T_{A}^{\bs k}=T_{A}^{\bs k}(T_{j}^{\ell_{j}})^{\#}U_{ja_{1}}^{*\ell_{j}k_{1}}\dots U_{ja_{m}}^{*\ell_{j}k_{m}}.    
\end{equation}
From (\ref{1 le tj and hatsh commutes with tak d}), we obtain
\begin{equation*}
T_{j}^{\ell_{j}}(T_{j}^{\ell_{j}})^{\#}T_{A}^{\bs k}= T_{j}^{\ell_{j}}T_{A}^{\bs k}(T_{j}^{\ell_{j}})^{\#}U_{ja_{1}}^{*\ell_{j}k_{1}}\dots U_{ja_{m}}^{*\ell_{j}k_{m}}.\end{equation*}

Then using (\ref{1 le tj and hatsh commutes with tak a}), we obtain
\begin{equation*}
T_{j}^{\ell_{j}}(T_{j}^{\ell_{j}})^{\#}T_{A}^{\bs k}= T_{A}^{\bs k}T_{j}^{\ell_{j}}(T_{j}^{\ell_{j}})^{\#}.    
\end{equation*}
\end{proof}

We are now in a position to present an alternative proof of the existence of a Wold-type decomposition for a doubly twisted near-isometry. Unlike the proof for Theorem \ref{1 th wold decom for dtni by induction}, which proceeds by induction, the arguments given below rely on the machinery developed in this section.
Although the statement remains unchanged, we restate it here for reader's convenience. 

\begin{theorem}\label{1th wold decompostion dtni by SOT}
Let $T=(T_{1},T_{2},\dots,T_{n})$ be a doubly twisted near-isometry. Then $T$ admits a unique Wold-type decomposition. Moreover, 
$\mathcal{H}=\bigoplus_{A\in I_{n}}\mathcal{H}_{A},$ where 
$$
\mathcal{H}_{A}=\bigoplus_{\bs k\in\mathbb{N}_{0}^{|A|}}T_{A}^{\bs k}(\bigcap_{\bs \ell\in\mathbb{N}_{0}^{n-|A|}} T_{I_{n}\setminus A}^{\bs \ell}\mathcal{W}_{A})
$$
with $\mathcal{W}_{A}=\bigcap\limits_{i\in A}kerT^{*},$ and $\cl H_{\emptyset}=\bigcap_{\bs \ell\in\mathbb{N}_{0}^{n}} T_{I_{n}}^{\bs \ell}\mathcal{\cl H}$.
\end{theorem}

\begin{proof} In view of Theorem \ref{1 th equivalent condition for wold}, the existence of Wold-type decomposition for $T$ follows once we show $\mathcal{H}_{{T}_{i},I}$ reduces $T_{j}$  for all $i,j\in I_{n}$. Recall that $\mathcal{H}_{T_{i},I}=\bigcap_{n\geq 0}T_{i}^{n}\mathcal{H}.$ Fix $i, j\in I_n$ with $i\ne j$. Then, 
$$
T_{j}(\bigcap_{n\geq 0} T_{i}^{n}\mathcal{H})=\bigcap_{n\geq 0}T_{j}T_{i}^{n}\mathcal{H}=\bigcap_{n\geq 0}U_{ji}^{n}T_{i}^{n}T_{j}\mathcal{H}=\bigcap_{n\geq 0}T_{i}^{n}T_{j}U_{ji}^{n}\mathcal{H}\subseteq\bigcap_{n\geq 0}T_{i}^{n}\mathcal{H},    
$$
and  
$$
{T_{j}}^{*}(\bigcap_{n\geq 0} T_{i}^{n}\mathcal{H})\subseteq\bigcap_{n\geq 0}{T_{j}}^{*}T_{i}^{n}\mathcal{H}=\bigcap_{n\geq 0}U_{ji}^{*n}T_{i}^{n}{T_{j}}^{*}\mathcal{H}=\bigcap_{n\geq 0}T_{i}^{n}{T_{j}}^{*}U_{ji}^{*n}\mathcal{H}\subseteq\bigcap_{n\geq 0}T_{i}^{n}\mathcal{H}.    
$$
This implies that $\cl H_{T_i,I}$ reduces $T_j$. Additionally, $\mathcal{H}_{T_{i},I}$ reduces $T_{i}.$ Thus, $\cl H_{{T}_{i},I}$ reduces $T_{j}$  for all $i,j\in I_{n}$; hence, 
$T$ admits Wold-type decomposition. The uniqueness of the decomposition follows from Remark \ref{1 remark unique rep of Wold for n-tuple of near-isometry}. 

Furthermore, Remark \ref{1 remark unique rep of Wold for n-tuple of near-isometry} described the summands in $\cl H=\bigoplus_{A\subseteq I_n} \cl H_A$ as 
$$
\mathcal{H}_{A}=\left(\bigcap_{i\in I_{n}\setminus A}\mathcal{H}_{T_{i},I}\right)\bigcap \left(\bigcap_{k\in A}\mathcal{H}_{T_{k},S}\right).
$$
The rest of the proof is about deriving the desired form for each $\cl H_A$. Using Proposition \ref{1 prop projection commutes in dtni}, we know that $\{P_{\mathcal{H}_{T_{i},S}}, P_{\mathcal{H}_{T_{j},I}}\}_{1\leq i,j\leq n}$ is a commuting family of projections on $\cl H$. Thus, we can rewrite $\cl H$ as  
\begin{equation} \label{eq1 wold sot}
\mathcal{H}_{A}=\left(\prod_{j\in I_{n}\setminus A}P_{\mathcal{H}_{T_{j},I}}\right)\left(\prod_{i\in A}P_{\mathcal{H}_{T_{i},S}}\right)\mathcal{H}.
\end{equation}

First, we analyze $\prod\limits_{i\in A}P_{\mathcal{H}_{T_{i},S}}\mathcal{H}$. Using Lemma \ref{1 le projection onto shift and invertible part},  
\begin{eqnarray*}
\prod_{i\in A}P_{\mathcal{H}_{T_{i},S}} &=& \prod_{i\in A}\left(SOT-\lim\limits_{r \to \infty} \sum_{k_{i}= 0}^{r}T_{i}^{k_{i}}(T_{i}^{k_{i}})^{\#}-T_{i}^{k_{i+1}}(T_{i}^{k_{i+1}})^{\#}\right)\\
&=& SOT-\lim\limits_{r \to \infty} \prod_{i\in A}\left(\sum_{k_{i}= 0}^{r}T_{i}^{k_{i}}(T_{i}^{k_{i}})^{\#}-T_{i}^{k_{i+1}}(T_{i}^{k_{i+1}})^{\#}\right)\\
&=& SOT-\lim\limits_{r \to \infty} \sum_{0\leq k_{1},\dots,k_{m}\leq r}\prod_{i\in A}\left(T_{i}^{k_{i}}(T_{i}^{k_{i}})^{\#}- T_{i}^{k_{i}+1}(T_{i}^{k_{i}+1})^{\#}\right), 
\end{eqnarray*}
where $m=|A|$. For notational convenience, let us define 
$$
P_r:= \sum_{0\leq k_{1},\dots,k_{m}\leq r}\prod_{i\in A}\left(T_{i}^{k_{i}}(T_{i}^{k_{i}})^{\#}- T_{i}^{k_{i}+1}(T_{i}^{k_{i}+1})^{\#}\right), 
$$
for $r\ge 0$. Then 
$$
\prod_{i\in A}P_{\mathcal{H}_{T_{i},S}}=SOT-\lim\limits_{r \to \infty} P_r.
$$

We shall show that $\{P_r\}_r$ is an increasing sequence of projections. Observe that, using Lemma \ref{1 le projection onto shift and invertible part}, 
$T_{i}^{k_{i}}(T_{i}^{k_{i}})^{\#}-T_{i}^{k_{i+1}}(T_{i}^{k_{i+1}})^{\#}$ is a projection onto 
$T_i^{k_i}(kerT_i^*)$, and using Lemma \ref{1 le power and hatsh of T_{i} and T_{j}} , $T_{i}^{k_{i}}(T_{i}^{k_{i}})^{\#}-T_{i}^{k_{i+1}}(T_{i}^{k_{i+1}})^{\#}$ commutes 
with $T_{j}^{k_{j}}(T_{j}^{k_{j}})^{\#}-T_{j}^{k_{j+1}}(T_{j}^{k_{j+1}})^{\#}$, whenever $i\ne j$. Thus, for each $(k_1, \dots, k_m)\in \bb N_0^m,$  
$$
\prod_{i\in A}\left(T_{i}^{k_{i}}(T_{i}^{k_{i}})^{\#}-T_{i}^{k_{i+1}}(T_{i}^{k_{i+1}})^{\#} \right)
$$
is a projection and 
\begin{eqnarray*}
\prod\limits_{i\in A}\left(T_{i}^{k_{i}}(T_{i}^{k_{i}})^{\#}-T_{i}^{k_{i+1}}(T_{i}^{k_{i+1}})^{\#} \right)\cl H &=&\bigcap_{i\in A} \left(T_{i}^{k_{i}}(T_{i}^{k_{i}})^{\#}-T_{i}^{k_{i+1}}(T_{i}^{k_{i+1}})^{\#}\right)\cl H\\
&=&\bigcap _{i\in A}T_{i}^{k_{i}} (kerT_{i}^{*}).
\end{eqnarray*}

Further, note that for distinct $\bs k=(k_{1},\dots, k_{n})$ and $\bs \ell=(\ell_{1},\dots, \ell_{n})$ 
$\bigcap _{i\in A}T_{i}^{k_{i}} (kerT_{i}^{*})\perp \bigcap _{i\in A}T_{i}^{\ell_{i}} (kerT_{i}^{*}),$ and therefore, 
$$
\prod\limits_{i\in A}\left(T_{i}^{k_{i}}(T_{i}^{k_{i}})^{\#}-T_{i}^{k_{i+1}}(T_{i}^{k_{i+1}})^{\#} \right) \perp
\prod\limits_{i\in A}\left(T_{i}^{\ell_{i}}(T_{i}^{\ell_{i}})^{\#}-T_{i}^{\ell_{i+1}}(T_{i}^{\ell_{i+1}})^{\#} \right),
$$ 
whenever $(k_1, \dots, k_m)\ne (\ell_1,\dots, \ell_m)$. This means, each $P_r$ is a projection and 
\begin{eqnarray*}
P_r\cl H &=& \bigoplus\limits_{1\le k_i\le r} \bigcap _{i\in A}T_{i}^{k_{i}} (kerT_{i}^{*})\\
&=&\bigoplus\limits_{\bs k\in I_r} T_A^{\bs k} \cl W_A \ \ \ ({\rm using \ Lemma \ \ref{1 le intersection of T_{i}W_{i}}}).
\end{eqnarray*}

Moreover, $\{P_r\}$ is an increasing sequence, by definition; therefore,  
\begin{eqnarray}\label{eq2 wold sot}
\prod_{i\in A} P_{\mathcal{H}_{T_i,S}}\cl H &=& SOT-\!\lim\limits_{r \to \infty} P_{r}\cl H\nonumber\\
&=& \overline{\bigcup\limits_{r\ge 0} \bigoplus\limits_{\bs k\in I_r} T_A^{\bs k} \cl W_A}\nonumber\\
&=& \bigoplus_{\bs k\in\mathbb{N}_{0}^{|A|}}T_{A}^{\bs k}(\mathcal{W}_{A}).
\end{eqnarray}

Using (\ref{eq2 wold sot}) in (\ref{eq1 wold sot}), and applying then Lemma \ref{1 le projection onto shift and invertible part} and Lemma \ref{1 le tj and hatsh commutes with tak}, we obtain 
 \begin{eqnarray*}
 \mathcal{H}_{A} &=&\left(\prod_{j\in I_{n}\setminus A}P_{\mathcal{H}_{T_{j},I}}\right)\bigoplus\limits_{\bs k\in \bb N_0^{|A|}} T_A^{\bs k} \cl W_A\\
 &=&  \bigoplus\limits_{{\bs k}\in \mathbb{N}_{0}^{|A|}} T_A^{\bs k}\left(\prod_{j\in I_{n}\setminus A}P_{\mathcal{H}_{T_{j}},I}\cl W_A\right).
 \end{eqnarray*}
 
Finally, using the fact $\cl W_A$ reduces $T_j$ for each $j\in I_n\setminus A$, and that if $T_j$ reduces a subspace $\cl M$, then $P_{\cl H_{T_j, I}}\cl M=\bigcap\limits_{r\ge 0}T_j^r\cl M$, we can show that 
$$
\prod_{j\in I_{n}\setminus A}P_{\mathcal{H}_{T_{j}},I}\cl W_A=\bigcap\limits_{{\bs \ell}\in \bb N_0^{n- |A|}}T_{I_n\setminus A}^{\bs \ell}\cl W_A.
$$
Hence, 
$$
\cl{H}_{A} = \bigoplus\limits_{{\bs k}\in \bb N_0^{|A|}} T_A^{\bs k}\left(\bigcap\limits_{{\bs \ell}\in \bb N_0^{n-| A|}}T_{I_n\setminus A}^{\bs \ell}\cl W_A\right).
$$
This completes the proof.
\end{proof}

\section{Unitary equivalence of doubly twisted near-isometries}\label{uni-equi}
In this section, we classify doubly twisted near-isometries up to unitary equivalence (Theorem \ref{1 th uni equiv of dtni}). 

\begin{definition}
An $n$-tuple $T=(T_{1},T_{2},\dots,T_{n})$ of operators on $\mathcal{H}$ is said to be unitarily equivalent to an $n$-tuple $S=(S_{1},S_{2},\dots, S_{n})$ of operators on $\mathcal{K}$ if there exists a unitary $U:\mathcal{H}\to \mathcal{K}$ such that  $UT_{i}U^{*}=S_{i}$ for all $1\leq i\leq n.$ We denote this by $T\cong S$.
\end{definition}
 
In the case of doubly twisted near-isometries, there is also 
an associated family of unitaries, and it is natural to require that the unitary $U$ intertwine the two families of unitaries as well. However, in this case, no additional assumption is needed, since this intertwining property, as we prove below, follows automatically.  

\begin{lemma}
Let $T=(T_{1},\dots ,T_{n})$ on $\cl H$ and $\widetilde{T}=(\widetilde{T}_{1},\dots ,\widetilde{T}_{n})$ on $\widetilde{\cl H}$ be doubly twisted near-isometries with respect to unitaries  $\{U_{ij}\}_{1\leq i\neq j\leq n}$ on $\cl H$ and $\{\widetilde{U}_{ij}\}_{1\leq i\neq j\leq n}$ on $\mathcal{\widetilde{H}}$, respectively. Suppose 
$V:\mathcal{H}\to \mathcal{\widetilde{H}}$ be a unitary operator such that $VT_{i}=\widetilde{T}_{i}V$ for $1\le i\le n$, then $VU_{ij}=\widetilde{U}_{ij}V$ for all $1\leq i\neq j\leq n$.
\end{lemma}

\begin{proof} Since $VT_{i}=\widetilde{T}_{i}V$, therefore $VT_i^{\#}={\widetilde{T_i}^{\#}}V$ for all i, where $T_i^{\#}=(T_i^*T_i)^{-1}T_i^*$ and 
$\widetilde{T}_i^{\#}=(\widetilde{T}_i^*\widetilde{T}_i)^{-1}\widetilde{T}_i^*$. Hence, 
\begin{equation*}
VU_{ij}=V(T_{i})^{\#}(T_{j})^{\#}T_{i}T_{j}=(\widetilde{T}_{i})^{\#}(\widetilde{T}_{j})^{\#}\widetilde{T}_{i}\widetilde{T}_{j}V = \widetilde{U}_{ij}V.    
\end{equation*}
\end{proof}

Before proceeding further, we first fix some notations. Recall that the Wold-type decomposition for a doubly twisted near-isometry $(T_1, \dots, T_n)$ on a Hilbert space $\cl H$ provides the orthogonal decomposition   
$\cl H=\bigoplus\limits_{A\subseteq I_n}\cl H_A$, where 
$$
\cl{H}_{A} = \bigoplus\limits_{{\bs k}\in \bb N_0} T_A^{\bs k}\left(\bigcap\limits_{{\bs \ell}\in \bb N_0^{I_n\setminus A}}T_{I_n\setminus A}^{\bs \ell}\cl W_A\right)
$$
with $\cl W_A=\bigcap\limits_{i\in A}kerT_i^*$, and with the convention that $\cl H_{\emptyset}=\bigcap\limits_{{\bs k}\in I_n}T_{I_A}^{\bs k}\cl H$.

For each $A\subseteq I_n$, we now fix notation  
 \begin{equation*}
\mathcal{D}_{A}(T)=\bigcap_{\bs \ell\in \mathbb{N}_{0}^{n-|A|} }T_{I_{n}\setminus A}^{\bs \ell}\mathcal{W}_{A},
 \end{equation*}
again with the understanding that $\cl D_{\emptyset}(T)=\cl H_{\emptyset}=\bigcap\limits_{{\bs k}\in I_n}T_{I_n}^{\bs k}\cl H.$ Thus, we can write  
\begin{equation}\label{wand}
\cl H_A=\bigoplus\limits_{{\bs k}\in \bb N_0^{|A|}} T_A^{\bs k}\cl D_A(T).
\end{equation}
When there is no ambiguity, we shall 
simply write $\cl D_A$ in place of $\cl D_A(T).$ The subspaces $\cl D_A$ were studied for tuples $T=(T_1, \dots, T_n)$ of isometries in \cite{jaydebmansirakshit2022} and \cite{pinto}, where they are referred to as $A$-wandering subspaces of $T$. We adopt this terminology in the setting of doubly twisted near-isometries as well. Since these subspaces play a vital role in our classification, we begin by recording some of their essential properties in our setting.

\begin{lemma}\label{reduce-wand}
Let $T=(T_{1},\dots,T_{n})$ be a doubly twisted near-isometry on $\mathcal{H}$. Then, for $A\subseteq I_n$, 
\begin{enumerate}
\renewcommand{\labelenumi}{(\roman{enumi})}
\item $\cl D_A$ reduces $T_j$ for each $j\in I_n\setminus A;$
\item $U_{ij}\cl D_A=\cl D_A$ for all $1\le i\ne j\le n$
\item $T_A^{*{\bs k}}T_{A}^{\bs k}\mathcal{D}_{A}=\mathcal{D}_{A}$ for all ${\bs k}\in \mathbb{N}_{0}^{|A|}$.
\end{enumerate}
\end{lemma}

\begin{proof} Let $A=\{a_1,\dots, a_m\}$ and $I_n\setminus A=\{q_1,\dots, q_{n-m}\}$. The proof works on similar lines if either $A$ or $I_n\setminus A$ is an empty set. Consider,
\begin{eqnarray*} 
T_{q_j}\cl D_A &=&T_{q_j} \left(\bigcap_{{\bs \ell} \in \bb{N}_{0}^{n-|A|} }T_{I_{n}\setminus A}^{\bs {\ell}}\cl {W}_{A}\right)\\
&=& \bigcap_{{\bs \ell}\in \bb{N}_{0}^{n-|A|}} T_{q_j}T_{I_n\setminus A}^{\bs {\ell}}\cl W_{A}\\
&=&\bigcap_{{\bs \ell}\in \bb N_0^{n-|A|}}T_{I_n\setminus A}^{\bs \ell}T_{q_j}U_{q_jq_1}^{\ell_1}\cdots U_{q_jq_{j-1}}^{\ell _{j-1}}U_{q_jq_{j+1}}^{\ell_{j+1}}\cdots U_{q_jq_{n-m}}^{\ell_{n-m}}\cl W_A\\
&\subseteq & \bigcap_{{\bs \ell} \in \bb{N}_{0}^{n-|A|} }T_{I_{n}\setminus A}^{\bs {\ell}}\cl {W}_{A} = \cl D_A, 
\end{eqnarray*}
since $\cl W_A$ is invariant under $T_{q_j}$ as well as each $U_{rs}$. In fact, $\cl W_A$ reduces both $T_j$ as well as each $U_{rs}$. Therefore, using similar set of arguments as above, we can also show that   
$$
T_{q_j}^*\cl D_A\subseteq \cl D_A.
$$
This proves $(i)$. Clearly,  $(ii)$ follows from the facts that $U_{rs}$ commutes with each $T_j$ and $U_{rs}\cl W_A=\cl W_A$.  

Lastly, we prove $(iii).$ First note that 
$\cl H_A$ reduces each $T_j$, which implies that $T_A^{*{\bs k}}T_A^{\bs k}\cl H_A\subseteq \cl H_A$ for every ${\bs k}\in \bb N_0^{|A|}.$ 
Then, in view of (\ref{wand}), $T_A^{*{\bs k}}T_A^{\bs k}\cl D_A\subseteq \cl D_A$ follows easily if we prove that $T_A^{*{\bs k}}T_A^{\bs k}\cl D_A$ is orthogonal $T_A^{\bs \ell}\cl D_A$ for all ${\bs \ell}\ne {\bs 0}\in \bb N_0^{|A|}.$ To this end,  let   
${\bs k}, {\bs \ell}\in \bb N_0^{|A|}, {\bs \ell}\ne {\bs 0}, \ \eta, \mu\in \cl D_A$, and consider  
\begin{eqnarray*}
\langle T_A^{*{\bs k}}T_A^{\bs k}\eta,T_A^{\bs \ell}\mu\rangle &=&\langle T_{A}^{\bs k}\eta, T_{A}^{\bs k}T_A^{\bs \ell}\mu\rangle\\
&=&   \langle T_A^{\bs k}\eta, T_{a_1}^{k_1+\ell_1}T_{a_2}^{k_2}\cdots T_{a_m}^{k_m}T_{a_2}^{\ell_2}\cdots T_{a_m}^{\ell_m}U_{a_2a_1}^{k_2\ell_1}\cdots U_{a_ma_1}^{k_m\ell_1}\mu\rangle\\
&=& \langle T_A^{\bs k}\eta , T_A^{\bs k+\bs \ell}R\mu\rangle, 
\end{eqnarray*}
where $R=(U_{a_2a_1}^{k_2\ell_1}\cdots U_{a_ma_1}^{k_m\ell_1})(U_{a_3a_2}^{k_3\ell_2}\cdots U_{a_ma_2}^{k_m\ell_2})\cdots (U_{a_m a_{m-1}}^{k_m\ell_{m-1}})$. Since $T_A^{\bs k}\cl D_A\perp T_A^{{\bs k}+{\bs \ell}}\cl D_A$ and $U_{ij}\cl D_A=\cl D_A$ for all $i\ne j$, therefore 
$$
\langle T_A^{*{\bs k}}T_A^{\bs k}\eta,T_A^{\bs \ell}\mu\rangle=0.
$$  

Thus, $T_A^{*{\bs k}}T_A^{\bs k}\cl D_A\subseteq \cl D_A$. Furthermore, $T_{A}^{*{\bs k}}T_{A}^{\bs k}$ is invertible and self adjoint; hence, 
$T_{A}^{*{\bs k}}T_{A}^{\bs k}\mathcal{D}_{A}=\mathcal{D}_{A}$.
\end{proof}

In the following result, we characterize unitarily equivalent doubly twisted near-isometries. The statement of Theorem \ref{1 th uni equiv of dtni} should be read in conjunction with 
Lemma \ref{reduce-wand}. 

\begin{theorem}\label{1 th uni equiv of dtni}
Let $T=(T_{1},\dots, T_{n})$ on $\cl H$ and $\widetilde{T}=(\widetilde{T}_1,\dots, \widetilde{T}_{n})$ on $\widetilde{\cl H}$ be doubly twisted near-isometries with associated unitaries $\{U_{ij}\}_{1\leq i\neq j\leq n}$ on $\cl H$ and $\{\widetilde{U}_{ij}\}_{1\leq i\neq j\leq n}$ on $\mathcal{\widetilde{H}}$, respectively. Then,  $T\cong \widetilde{T}$ if and only if for each $A\subseteq I_{n}$, there exists a unitary $V_{A}:\mathcal{D}_{A}(T)\to \mathcal{D}_{A}(\widetilde{T})$ such that on $\cl D_A(T)$ the following hold
\begin{enumerate}
\renewcommand{\labelenumi}{(\roman{enumi})}
 \item $V_{A}\bigl((T_{A}^{\bs k})^{*}T_{A}^{\bs k}\bigr)=\bigl((\widetilde{T}_{A}^{\bs k})^{*}\widetilde{T}_{A}^{\bs k}\bigr)V_{A}$;
 \item $V_{A}T_{j}=\widetilde{T}_{j}V_{A}$ for $j\in I_{n}\setminus A$;
 \item $V_{A}U_{ij}=\widetilde{U}_{ij}V_{A}$ for $1\le i\ne j\le n$.
\end{enumerate}
 \end{theorem}
 
\begin{proof}
Suppose $V:\mathcal{H}\to \mathcal{\widetilde{H}}$ be a unitary such that $VT_{i}=\widetilde{T}_{i}V$ for each $i$. It readily follows that $V\mathcal{D}_{A}(T)=\mathcal{D}_{A}(\widetilde{T})$. Denote $V_{A}=V|_{\mathcal{D}_{A}(T)}$. Then, $V_A:\cl D_A(T) \to \cl D_A(\widetilde{T})$ is a unitary. 

Next, Lemma \ref{reduce-wand} proves that, for $j\notin A,$ the subspaces $\cl D_A(T)$ and $\cl D_A(\widetilde{T})$ reduce $T_j$ and $\widetilde{T}_j$, respectively. Further, the lemma also asserts that $U_{ij}\cl D_A(T)=\cl D_A(T)$ and $\widetilde{U}_{ij}\cl D_A(\widetilde{T})=\cl D_A(\widetilde{T})$ for every $1\le i\ne j\le n$. Thus,  
$$
V_AT_j=\widetilde{T}_jV_A \ \  \ {\rm and} \ \ \ V_AU_{rs}=\widetilde{U}_{rs}V_A \ \ \ {\rm on} \ \cl D_A(T),
$$ 
for every $j\notin A$ and $1\le r\ne s\le n.$ This establishes $(ii)$ and $(iii)$. Now, to prove $(i)$, note that Lemma \ref{reduce-wand} also asserts that $\bigl((T_{A}^{\bs k})^{*}T_{A}^{\bs k}\bigr)\cl D_A(T)\subseteq \cl D_A(T)$. Consequently,  
$$
V_{A}\bigl((T_{A}^{\bs k})^{*}T_{A}^{\bs k}\bigr)(\eta)=V\bigl((T_{A}^{\bs k})^{*}T_{A}^{\bs k}\bigr)(\eta)=\bigl((\widetilde{T}_{A}^{\bs k})^{*}\widetilde{T}_{A}^{\bs k}\bigr)V(\eta)
= \bigl((\widetilde{T}_{A}^{\bs k})^{*}\widetilde{T}_{A}^{\bs k}\bigr)V_A(\eta)
$$
for all $\eta\in \cl D_A(T)$, which establishes $(i).$ 

For the converse, suppose, for each $A\subseteq I_n$, we have a unitary $V_A:\cl D_A(T)\to \cl D_A(\widetilde{T})$ that satisfies the conditions $(i), \ (ii)$ and $(iii)$.  
We must show that $T=(T_1, \dots, T_n)$ and $\widetilde{T} = (\widetilde{T}_1, \dots, \widetilde{T}_n)$ are unitarily equivalent.  

Firstly, since $(T_1, \dots, T_n)$ and $(\widetilde{T}_1, \dots, \widetilde{T}_n)$ are both doubly twisted near-isometries, they admit Wold-type decompositions, which means we have the decompositions  
$$
\cl H=\bigoplus\limits_{A\subseteq I_n}\cl H_A \ \ {\rm and} \ \ \widetilde{\cl H} = \bigoplus\limits_{A\subseteq I_n} \widetilde{\cl H}_A,
$$ 
where $\cl H_A=\bigoplus\limits_{{\bs k}\in \bb N_0^{|A|}} T_A^{\bs k}\cl D_A(T)$ and $\widetilde{\cl H}_A=\bigoplus\limits_{{\bs k}\in \bb N_0^{|A|}}\widetilde{T}_A^{\bs k}\cl D_A(\widetilde{T}).$

We shall construct the desired unitary from $\cl H$ onto $\widetilde{\cl H}$ by first defining unitaries between $\cl H_A$ and $\widetilde{\cl H}_A$, and then taking their orthogonal direct sum. For $A=\emptyset$, $\cl D_A(T)=\cl H_{\emptyset}$ and $\cl D_A(\widetilde{T})=\widetilde{\cl H}_{\emptyset}$. 
 Define $U_\emptyset = V_\emptyset$. Then, $U_\emptyset : \cl H_\emptyset \to \widetilde{\cl H}_\emptyset$ is a unitary, and by $(ii)$, 
\begin{equation}\label{intertwine4}
U_\emptyset T_j=\widetilde{T}_jU_\emptyset  \ \ \ {\rm on} \ \cl H_A
\end{equation}
for all $1\le j\le n.$

Now, suppose $A\ne \emptyset$, and let $\Lambda_{A, {\bs k}}$ and $\widetilde{\Lambda}_{A, {\bs k}}$ be the unitary operators appearing in the polar decompositions of the invertible 
operator $T_A^{\bs k}: \cl D_A(T)\to T_A^{\bs k}\cl D_A(T)$ and $\widetilde{T}_A^{\bs k}: \cl D_A(\widetilde{T}) \to \widetilde{T}_{A}^{k}\cl D_A(\widetilde{T})$, respectively, that is, 
\begin{equation}\label{polar}
T_A^{\bs k}= \Lambda_{A, {\bs k}}\left| T_A^{\bs k}|_{\cl D_A(T)} \right| \ {\rm and} \ \widetilde{T}_A^{\bs k}= 
\widetilde{\Lambda}_{A, {\bs k}}\left| \widetilde{T}_A^{\bs k}|_{\cl D_A(\widetilde{T})} \right|.
\end{equation}

Then, for each fixed $A\subseteq I_n$ and ${\bs k}\in \bb N_0^{|A|}$, 
$$
U_{A, {\bs k}} := \widetilde{\Lambda}_{A, {\bs k}} V_A \Lambda_{A, {\bs k}}^{*}
$$
defines a unitary from $T_A^{\bs k}\cl D_A(T)\to \widetilde{T}_A^{\bs k}\cl D_A(\widetilde{T})$; therefore, $U_A=\bigoplus\limits_{\bs k\in \bb N_0^{|A|}} U_{A, {\bs k}}$ gives a unitary from $\cl H_A\to \widetilde{\cl H}_A$. Note that $U_A=V_A$ on $\cl D_A(T)$. Finally, 
$$
U=\bigoplus\limits_{A\subseteq I_n} U_A
$$
is a unitary from $\cl H$ onto $\widetilde{\cl H},$ where $U_\emptyset = V_\emptyset.$ We claim that $U$ is the desired unitary. We now turn to proving that $UT_i=\widetilde{T}_iU$ for all $1\le i\le n.$ Fix $\emptyset \ne A\subseteq I_n$. Lemma \ref{reduce-wand} implies that $\cl D_A(T)$ is invariant under $(T_A^{\bs k})^{*}T_A^{\bs k}$, which implies that, for each $\eta\in \cl D_A(T)$, we have  
$$
\left| T_A^{\bs k}|_{\cl D_A(T)}\right|^2(\eta) = \left((T_A^{\bs k})^{*}T_A^{\bs k}\right)(\eta),
$$
which further, along with $(i)$, implies that 
\begin{equation}\label{square-root}
V_A \left| T_A^{\bs k}|_{\cl D_A(T)}\right| =  \left| \widetilde{T}_A^{\bs k}|_{\cl D_A(\widetilde{T})}\right| V_A. 
\end{equation}

Then, for any ${\bs k}\in \bb N_0^{|A|}$, we get   
\begin{eqnarray}
U_AT_A^{\bs k}(\eta) &=& U_{A, {\bs k}}T_A^{\bs k}(\eta)\nonumber\\
&=& \widetilde{\Lambda}_{A, {\bs k}} V_A \Lambda_{A, {\bs k}}^{*}T_A^{\bs k}(\eta)\nonumber\\
&=&  \widetilde{\Lambda}_{A, {\bs k}} V_A \left| T_A^{{\bs k}}|_{_{\cl D_A(T)}}\right|(\eta)\nonumber\\
& = &  \widetilde{\Lambda}_{A, {\bs k}} \left| \widetilde{T}_A^{{\bs k}}|_{_{\cl D_A(\widetilde{T})}}\right|V_A(\eta) \ \ ({\rm using} \ (\ref{square-root})) \nonumber\\
&=& \widetilde{T}_A^{\bs k} V_A(\eta) \label{intertwine}\\
&=& \widetilde{T}_A^{\bs k} U_A(\eta)\nonumber,
\end{eqnarray}
since $U_A=V_A$ on $\cl D_A(T).$ Now, let $A=\{a_1,\dots, a_m\}$. Then, for ${\bs k}\in \bb N_0^{|A|}$ and $\eta\in \cl D_A(T)$, we have 
\begin{eqnarray}
U_AT_{a_i}(T_A^k\eta ) &=& U_AT_A^{{\bs k}+e_i}U_{a_ia_1}^{k_1}\cdots U_{a_ia_{i-1}}^{k_{i-1}}(\eta)\nonumber\\
&=& \widetilde{T}_A^{{\bs k}+e_i}V_A U_{a_ia_1}^{k_1}\cdots U_{a_ia_{i-1}}^{k_{i-1}}(\eta)  \ \ \  \ ({\rm using} \ (\ref{intertwine}))\nonumber\\
&=& \widetilde{T}_{a_i}\widetilde{T}_A^{\bs k} \widetilde{U}_{a_1a_i}^{k_1}\cdots \widetilde{U}_{a_{i-1}a_i}^{k_{i-1}}V_AU_{a_ia_1}^{k_1}\cdots U_{a_ia_{i-1}}^{k_{i-1}}(\eta)\nonumber\\
& = & \widetilde{T}_{a_i}\widetilde{T}_A^{\bs k}V_A(\eta)  \ \ \ \ ({\rm using} \ (iii) \ {\rm and} \ U_{lr}=U_{rl}^*)\nonumber\\
&=& \widetilde{T}_{a_i} U_AT_A^{\bs k}(\eta) \ \ \  \ ({\rm using} \ (\ref{intertwine})) \label{intertwine1}.
\end{eqnarray}

Let $I_n\setminus A=\{q_1, \dots, q_{n-m}\}$. Then, for any ${\bs k}\in \bb N_0^{|A|}$ and $\eta\in \cl D_A(T)$, we have 
\begin{eqnarray}
U_AT_{q_j}T_A^{\bs k}(\eta) &=& U_AT_A^{\bs k}T_{q_j}U_{q_ja_1}^{k_1}\cdots U_{q_ja_m}^{k_m}(\eta)\nonumber\\
&=& \widetilde{T}_A^{\bs k}V_AT_{q_j}U_{q_ja_1}^{k_1}\cdots U_{q_ja_m}^{k_m}(\eta) \ \ \ \ ({\rm using} \ (\ref{intertwine}))\nonumber\\
&=& \widetilde{T}_A^{\bs k}\widetilde{T}_{q_j}V_AU_{q_ja_1}^{k_1}\cdots U_{q_ja_m}^{k_m}(\eta) \ \ \ \ ({\rm using} \ (ii))\nonumber\\
&=& \widetilde{T}_{q_j} \widetilde{T}_A^{\bs k}\widetilde{U}_{a_1q_j}^{k_1}\cdots \widetilde{U}_{a_mq_j}^{k_m}V_AU_{q_ja_1}^{k_1}\cdots U_{q_ja_m}^{k_m}(\eta)\nonumber\\
&=& \widetilde{T}_{q_j}\widetilde{T}_A^{\bs k}V_A(\eta)   \ \ \ \ ({\rm using} \ (iii) \ {\rm and} \ U_{lr}=U_{rl}^*)\nonumber\\
&=& \widetilde{T}_{q_j}U_AT_A^{\bs k}(\eta)  \ \ \  \ ({\rm using} \ (\ref{intertwine}))\label{intertwine2}. 
\end{eqnarray}
  
Hence, (\ref{intertwine1}) and (\ref{intertwine2}) together imply that, for every $\emptyset \ne A\subseteq I_n$, we have  
$$
U_AT_{q_j}=\widetilde{T}_{q_j}U_A  \ \ {\rm and} \ \  U_AT_{a_i}=\widetilde{T}_{a_i}U_A \ \ \ {\rm on} \ \cl H_A
$$
for all $a_i\in A$ and $q_j\in I_n\setminus A$, that is, 
\begin{equation}\label{intertwine3} 
U_AT_j=\widetilde{T}_j U_A
\end{equation}
for all $1\le j\le n$, whenever $A\ne \emptyset.$

Finally, combining (\ref{intertwine3}) and (\ref{intertwine4}), together with the fact each $\cl H_A$ reduces every $T_j$, we conclude that 
$$
UT_j=\widetilde{T}_j U  \ \ \ {\rm on} \ \cl H
$$ 
for every $1\le j\le n$. This completes the proof. 
\end{proof}
 
A classification of unitary equivalence for  doubly twisted isometries is given in \cite{jaydebmansirakshit2022}. The same paper defines the notion of $A$-wandering data  associated with an $A$-wandering subspace $\cl D_A(T)$, which forms the basis of this classification. 

Recall that for a doubly twisted isometry $T=(T_{1},T_{2},\dots ,T_{n})$ on $\cl H$ (and hence automatically a doubly twisted near-isometry) and a fixed 
subset $A\subseteq I_{n}$, the $A$-wandering subspace of $T$ is defined as  
\begin{equation*}
\mathcal{D}_{A}(T)=\bigcap_{\bs \ell \in \mathbb{N}_{0}^{n-|A|}}T_{I_{n}\setminus A}^{\bs \ell}\mathcal{W}_{A},  
\end{equation*}
where for $A=\phi$, we follow the convention that  
$$
\mathcal{D}_{A}(T)=\bigcap_{\bs \ell\in \mathbb{N}_{0}^{I_{n}}}T_{I_{n}}^{\bs \ell}\mathcal{H}.    
$$
Suppose we write $I_n\setminus A=\{q_1, \dots, q_{n-m}\}$ with $m=|A|$. The $A-$wandering data of $T$ is defined as 
$$
wd_{T}(A)=(I|_{\mathcal{D}_{A}(T)},T_{q_{1}}|_{\mathcal{D}_{A}(T)},\dots ,T_{q_{n-m}}|_{\mathcal{D}_{A}(T)}),
$$
a $n-m+1$-tuple of operators on $\cl D_A(T)$, where $I$ is the identity operator.  

We continue to use this same notion of a $A$-wandering data for a $A$-wandering subspace $\cl D_A(T)$ associated with a doubly twisted near-isometry. The following definition of unitary equivalence of wandering data is inspired from \cite{jaydebmansirakshit2022}, where it is introduced for doubly twisted 
isometries; here, we extend it to the setting of doubly twisted near-isometries. Of course, every doubly twisted isometry is a doubly twisted near-isometry; therefore, our definition reduces to theirs in the case of a doubly twisted isometry.

\begin{definition}
Let $T=(T_1,T_2,\dots T_n)$ and $\widetilde{T}=(\widetilde{T}_{1},\widetilde{T}_{2},\dots \widetilde{T}_{n})$ be doubly twisted near-isometries on $\mathcal{H}$ and $\widetilde{\mathcal{H}}$, respectively.
We call $wd_{T}(A)\cong wd_{\widetilde{T}}(A)$ if there exists a unitary operator $V_{A}:\mathcal{D}_{A}(T)\to \mathcal{D}_{A}(\widetilde{T})$ such that the following holds:
\begin{enumerate}
\item[(a)] $V_{A}T_{j}=\widetilde{T}_{j}V_{A}$ for $j\in I_{n}\setminus A$.
\item[(b)] $V_{A}U_{ij}=\widetilde{U}_{ij}V_{A}$ for $1\leq i\neq j\leq n$.
\end{enumerate}
\end{definition}    

We now recall the characterization of unitary equivalence for doubly twisted isometries from \cite{jaydebmansirakshit2022}. 

\begin{theorem}\label{wand-data}
Let $T=(T_1,T_2,\dots T_n)$ and $\widetilde{T}=(\widetilde{T}_{1},\widetilde{T}_{2},\dots \widetilde{T}_{n})$ be doubly twisted isometries on $\mathcal{H}$ and $\widetilde{\mathcal{H}}$, respectively. Then, $T\cong \widetilde{T}$ if and only if $wd_{T}(A)\cong wd_{\widetilde{T}}(A)$ for all $A\subseteq I_{n}$.
\end{theorem}

As a consequence of Theorem \ref{1 th uni equiv of dtni}, the above result follows by specializing to the isometric case. Indeed, if $T=(T_1,\dots,T_n)$ is a doubly twisted isometry, then $T_A^{\boldsymbol{k}}$ is an isometry for every nonempty $A\subseteq I_n$ and every $\boldsymbol{k}\in \mathbb{N}_0^{|A|}$. Consequently, $(T_A^{\boldsymbol{k}})^*T_A^{\boldsymbol{k}}=I$, and hence condition $(i)$ of Theorem \ref{1 th uni equiv of dtni} is automatically satisfied.

Moreover, conditions $(ii)$ and $(iii)$ of Theorem \ref{1 th uni equiv of dtni} reduce precisely to the requirement that $wd_{T}(A)\cong wd_{\widetilde{T}}(A)$ for all $A\subseteq I_n$. Therefore, in the special case of a doubly twisted isometry, our classification theorem (Theorem \ref{1 th uni equiv of dtni}) coincides with Theorem \ref{wand-data} of \cite{jaydebmansirakshit2022}.

However, for the general case of a doubly twisted near-isometry, Example \ref{ex-wand} illustrates that condition $(i)$ of Theorem \ref{1 th uni equiv of dtni} is essential, that is, $wd_{T}(A)\cong wd_{\widetilde{T}}(A)$ does not, in general, imply $T\cong \widetilde{T}.$ 

\begin{example}\label{ex-wand}
Let $\mathcal{H}=H^{2}(\bb D)\otimes H^2(\bb D)$, and let $\widetilde{M}_{z}$ be the weighted shift on $H^2(\bb D)$ with weights $w_n =\left(\frac{1}{3}+\frac{1}{3^{n+1}}\right)$.
Define 
$$
T_{1}=M_{z}\otimes I, \ \ T_{2}=I\otimes M_{z}, \ \ \widetilde{T}_{1}=\widetilde{M}_{z}\otimes I, \ \ {\rm and} \ \ \widetilde{T}_{2}=I\otimes \widetilde{M}_{z}. 
$$

Since $M_z$ is an isometry and $\widetilde{M}_{z}$ is a bounded contractions, it is straightforward to verify that $T_{1},T_{2}$ are isometries, and $ \widetilde{T}_{1}$, and $\widetilde{T}_{2}$ are bounded below contractions.  

Further, observe that $ker \widetilde{T}_{1}^{*}=\bb C\otimes H^2(\bb D)$. For each fixed $n\ge 0$, the subspace $\widetilde{T}_1^n(ker \widetilde{T}_1^*)$ is 
contained in the closed linear span of $\{z^n\otimes z^m: m\ge 0\}$, whereas $\widetilde{T}_1^{n+1}\cl H$ is in contained in the closed linear span of $\{z^{n+1}\otimes z^m: m\ge 0\}$. Consequently, 
$$\widetilde{T}_1^n(ker T_1^*) \perp \widetilde{T}_1^{n+1}\cl H.
$$ 
Hence, $\widetilde{T}_1$ is a near-isometry on $\cl H$. By similar argument $\widetilde{T}_{2},T_{1}, T_{2}$ are also near-isometries on $\cl H$. In addition, since 
$$
T_1T_2=T_2T_1, \  \ T_1^*T_2=T_2T_1^* \ \ \widetilde{T}_1\widetilde{T}_2=\widetilde{T}_2\widetilde{T}_1, \ {\rm and} \ \widetilde{T}_1^*\widetilde{T}_2=\widetilde{T}_2 \widetilde{T}_1^* 
$$
both $T=(T_1, T_2)$ and $\widetilde{T}=(\widetilde{T}_1,\widetilde{T}_{2})$ are doubly twisted near-isometries with respect to the identity operator on 
$\cl H$. 

Note that $||T_1||=1$ and $||\widetilde{T}_1||=2/3$; hence, $T\ncong \widetilde{T}.$ Nevertheless, we show that $wd_{T}(A)\cong wd_{\widetilde{T}}(A)$ for every 
$A\subseteq I_2.$ For this, we calculate the $A$-wandering subspace of $T$ and of $\widetilde{T}$ for every $A\subseteq \{1, 2\}$.  A direct computation shows that 

\begin{enumerate}
\renewcommand{\labelenumi}{(\roman{enumi})}
\item $\cl W_{\{1\}} = \widetilde{\mathcal{W}}_{\{1\}}=\bb C\otimes H^2(\bb D)$;
\item $\mathcal{W}_{\{2\}}=\widetilde{\mathcal{W}}_{\{2\}}=H^2(\bb D) \otimes \mathbb{C}$;
\item $\cl W_{\{1,2\}}=\mathcal{W}_{1}\cap \mathcal{W}_{2}=\mathbb{C}\otimes \mathbb{C}=\widetilde{\cl W}_{1}\cap \widetilde{\cl W}_{2}=\widetilde{\cl W}_{\{1,2\}}$.
\end{enumerate}

Then, 
\begin{eqnarray*}
\cl D_{\emptyset}(T) &=& \bigcap_{n, m\geq 0}(M_{z}^{n}\otimes M_z^m)(H^{2}(\bb D)\otimes H^2(\bb D)) =\{0\},\\
\cl D_{\{1\}}(T) &=& \bigcap_{n\geq 0}(I\otimes M_{z}^{n})(\bb C \otimes H^2(\bb D)) =\{0\},\\
\cl D_{\{2\}}(T) &=& \bigcap_{n\geq 0}(M_{z}^{n}\otimes I)(\bb C \otimes H^2(\bb D)) =\{0\},\\
\cl D_{\{1,2\}}(T) &=& \bb C \otimes \bb C.
\end{eqnarray*}
Smilarly, 
$$
\cl D_{\emptyset}(\widetilde{T})  =\{0\}, \ \ \cl D_{\{1\}}(\widetilde{T}) = \{0\}, \ \ \cl D_{\{2\}}(\widetilde{T}) =\{0\}, \ {\rm and} \ \cl D_{\{1,2\}}(\widetilde{T})= \bb C \otimes \bb C.
$$

Hence, it follows that $wd_{T}(\phi)\cong wd_{\widetilde{T}}(\phi)$ trivially for every proper subset $A$ of $\{1,2\}.$ Finally, 
$wd_{T}(\{1,2\})\cong wd_{\widetilde{T}}(\{1,2\})$ follows from the facts that $\cl D_{\{1,2\}}(T)=\bb C\otimes \bb C = \cl D_{\{1,2\}}(\widetilde{T})$, and 
$I_2\setminus A=\emptyset$.
\end{example}

\section{An analytic model for doubly twisted near-isometries }\label{doubly-ana}
In this section, we construct an analytic model for a doubly twisted near-isometry. 
In Theorem~\ref{1 th analytic model of shift ni}, an analytic model for the shift part of a near-isometry is obtained in terms of an operator-valued weighted shift operator acting on a vector-valued Hardy space over $\bb D$. Motivated by this construction, we now develop an analytic model for a doubly twisted near-isometry. Since we are dealing with a tuple of operators, this necessitates a multivariable analogue of the operator-valued weighted shift operator acting on a vector-valued Hardy space over a polydisc, namely, an operator-valued multishift. 
  
\begin{definition} Let $m$ be a positive integer and let $\{\Gamma_{\bs k, i}: {\bs k}\in \bb N_0^m, \ 1\le i\le m\}$ be a family of uniformly bounded operators on a Hilbert space $\cl H$. Then an operator-valued multishift on $H^2_{\cl H}(\bb D^m)$ is the tuple $(M_1, \dots, M_m)$ of operators defined by 
$$
M_i (z^{\bs k}\otimes \eta)=M_{z_i}(z^{\bs k})\otimes \Gamma_{{\bs k}, i}(\eta).
$$
\end{definition}

In what follows, we first construct a model for a doubly twisted near-isometry on $\cl H_A$ (as appearing in (\ref{wand})) using operator-valued multishifts and operator-valued diagonal operators (Theorem \ref{1 th unitary eq of dtni on ha}). We then glue these models together to obtain an analytic model for the doubly twisted near-isometry on the entire space (Theorem \ref{1 th unitary eq of dtni on h}).  

\begin{theorem}\label{1 th unitary eq of dtni on ha} 
Let $T=(T_{1},T_{2},\dots,T_{n})$ be a doubly twisted near-isometry on $\cl H$, and let $\mathcal{H}=\bigoplus\limits_{A\subseteq I_{n}}\mathcal{H}_{A}$ be the Wold-type decomposition of $\cl H$. Fix $\emptyset \ne A\subseteq I_{n}$. Then there exist an $n$-tuple $M_A=(M_{A,1}, \dots, M_{A,n})$ of operator on $H^2_{\cl D_A}(\bb D^{|A|})$ along with a unitary $U_{A}:\mathcal{H}_{A}\to H_{\mathcal{D}_{A}}^{2}(\mathbb{D}^{|A|})$ 
such that 
$$U_{A}T_{s}|_{\mathcal{H}_{A}}U_{A}^{*}=M_{A,s}
$$ 
for each $1\le s\le n$, where 
$$
M_{A,s}(z^{\bs k}\otimes \eta) =
\begin{cases}
M_{z_i}(z^{\bs k})\otimes \Lambda_{A, {\bs k}+e_i}^*T_{a_{i}}\Lambda_{A, {\bs k}}(\eta) & \text{if } \ s=a_{i}, 1\leq i\leq m \\
z^{\bs k}\otimes \Lambda_{A, {\bs k}}^*T_{q_{j}}\Lambda_{A, \bs k}\eta& \text{if } s=q_{j}, 1\leq j\leq n-m
\end{cases}
$$
and $\Lambda_{A, {\bs k}}: \cl D_A\to T_A^{\bs k}(\cl D_A)$ is a unitary for each ${\bs k}\in \bb N_0^{|A|}\}$.
\end{theorem}

\begin{proof} Let $A\ne \emptyset$. For each fixed ${\bs k}\in \bb N_0^{|A|}$, $T_A^{\bs k}: \cl D_A \to T_A^{\bs k}\cl (\cl D_A)$ is invertible. Thus, its polar decomposition gives a unitary map 
$\Lambda_{A, \bs k}:  \cl D_A \to T_A^{\bs k}\cl (\cl D_A)$. This allows us to define a unitary operator $U_{A}:\mathcal{H}_{A}\to H_{\mathcal{D}_{A}}^{2}(\mathbb{D}^{|A|})$ by  
$$
U_{A}(T_{A}^{\bs k}\eta)=z^{\bs k}\otimes \Lambda_{A, \bs k}^*(T_{A}^{\bs k}\eta).
$$
for every $\eta\in \cl D_A$ and ${\bs k}\in \bb N_0^{|A|}.$ We shall show that $U_A$ is the required unitary. Let $A=\{a_1, \dots, a_m\}$ with $a_1<\cdots<a_m$.  Then, for $\eta\in \cl D_A$ and ${\bs k}\in \bb N_0^{|A|}$,  
\begin{eqnarray*}
 U_{A}T_{a_{i}}U_{A}^{*}(z^{\bs k}\otimes \eta) &=& U_{A}T_{a_{i}}\Lambda_{A, \bs k}(\eta)\\
    &=& U_AT_{a_i}T_A^{\bs k}(\mu) \ \ \ ({\rm taking} \ \Lambda_{A, \bs k}(\eta)=T_{A}^{\bs k}\mu \ {\rm for \ some} \ \mu\in \cl D_A)\\
    &=& U_{A}T_{a_{i}}(T_{a_{1}}^{k_{1}}\dots T_{a_{i-1}}^{k_{i-1}}T_{a_{i}}^{k_{i}}\dots T_{a_{m}}^{k_{m}}\mu)\\
  &=& U_{A}(T_{a_{1}}^{k_{1}}\dots T_{a_{i-1}}^{k_{i-1}}T_{a_{i}}^{k_{i}+1}\dots T_{a_{m}}^{k_{m}})(U_{a_{i}a_{1}}^{k_{1}}U_{a_{i}a_{2}}^{k_{2}}\dots U_{a_{i}a_{i-1}}^{k_{i-1}}\mu)\\
&=&z^{{\bs k}+e_i}\otimes \Lambda_{A, \bs k+e_{i}}^*T_A^{{\bs k}+e_i}(U_{a_{i}a_{1}}^{k_{1}}\dots U_{a_{i}a_{i-1}}^{k_{i}-1}\mu) \ {\rm as} \ U_{rs}\cl D_A\subseteq \cl D_A\\
& =& z^{{\bs k}+e_i}\otimes \Lambda_{A, \bs k+e_i}^* T_{a_{i}}T_{A}^{\bs k}\mu\\
& =& z^{{\bs k}+e_{i}}\otimes \Lambda_{A, \bs k+e_i}^*T_{a_{i}}\Lambda_{A, \bs k}\eta.
\end{eqnarray*}

Now, we shall compute $U_AT_jU^*_A$ for $j\in I_n\setminus A$. Let $I_n\setminus A=\{q_1, \dots, q_{n-m}\},$ and for ${\bs k}\in \bb N_0^{|A|}$ and $\eta\in \cl D_A$, consider 
\begin{eqnarray*}
U_AT_{q_j}U_A(z^{\bs k}\otimes \eta) &=& U_{A}T_{q_j}\Lambda_{A, \bs k}(\eta)\\
&=& U_AT_{q_j}T_A^{\bs k}(\mu) \ \ \ ({\rm taking} \ \Lambda_{A, \bs k}(\eta)=T_{A}^{\bs k}\mu \ {\rm for \ some} \ \mu\in \cl D_A)\\
&=& U_{A}T_A^{\bs k}T_{q_j}(U_{q_ja_{1}}^{k_{1}}U_{q_ja_{2}}^{k_{2}}\dots U_{q_ja_{m}}^{k_m}\mu) \\
&=&z^{{\bs k}}\otimes \Lambda_{A, \bs k}^*T_{q_j}T_A^{\bs k}\mu)\\
& =& z^{\bs k}\otimes \Lambda_{A, \bs k}^*T_{q_j}\Lambda_{A, \bs k}\eta,
\end{eqnarray*}
since $\cl D_A$ is invariant under $T_{q_j}$ and each $U_{rs}.$   

Hence,
$$
U_{A}T_{a_{i}}U_{A}^{*}(z^{\bs k}\otimes \eta) = M_{z_i}z^{\bs k}\otimes \Lambda_{A, \bs k+e_i}^*T_{a_{i}}\Lambda_{A, \bs k}\eta
$$
and 
$$
U_AT_{q_j}U_A(z^{\bs k}\otimes \eta)=z^{\bs k}\otimes \Lambda_{A, \bs k}^*T_{q_j}\Lambda_{A, \bs k}\eta
$$
for all $\eta \in \cl D_A$ and ${\bs k}\in \bb N_0^{|A|}.$ This completes the proof. 
\end{proof}

\begin{remark}
For $A=\emptyset,$ recall that by convention $\cl D_A=\cl H_\emptyset$, and in this case, for the sake of uniformity of notation, we shall continue to use 
the symbol $H^2_{\cl D_A}(\bb D^{|A|})$ throughout the rest of the paper to denote $\cl H_{\emptyset}$. We then set $M_{\emptyset, s}=T_s$ for all $1\le s\le n$. With this convention, the statement of Theorem \ref{1 th unitary eq of dtni on ha} remains valid for the case $A=\emptyset$, although in a trivial manner. 
\end{remark}

Before proceeding to construct a model for a doubly twisted near-isometry on the full space, we highlight some salient features of the weight operators appearing in the definition of the operators $M_{A, s}$ in Theorem \ref{1 th unitary eq of dtni on ha}. Let $\emptyset \ne A\subseteq I_n$ and write $A=\{a_1, \dots, a_m\}$ of $I_n$ with $a_1<\cdots <a_m$. For ${\bs k}\in \bb N_0^m$ and $1\le s\le n$, define  
$$
\Gamma^A_{{\bs k},s} =
\begin{cases}
\Lambda_{A, {\bs k}+e_i}^*T_{a_{i}}\Lambda_{A, {\bs k}} & \text{if } \ s=a_{i}, 1\leq i\leq m \\
\Lambda_{A, \bs k}^* T_{s}\Lambda_{A, \bs k} & \text{if } s\notin A
\end{cases},
$$
where the unitaries $\Lambda_{A, {\bs k}}$ are as defined in Theorem \ref{1 th unitary eq of dtni on ha}. It is straightforward to verify that   
each operator $\Gamma^A_{{\bs k}, s}$ is invertible on $\cl D_A$, and that there exists a constant $c>0$ such that 
$$ 
c ||\eta || \le ||\Gamma^A_{{\bs k}, s} \eta|| \le ||\eta||
$$
for all $\eta\in \cl D_A, \ {\bs k}\in \bb N_0^m$ and $1\le s\le n$. 

Consequently, the $m$-tuple $(M_{A, a_1}, \dots, M_{A, a_m})$ defines an operator-valued multishift on $H^2_{\cl D_A}(\bb D^m)$, while each operator 
$M_{A, s}$ with $s\notin A$ is an invertible block-diagonal operator. 

\begin{theorem}\label{1 th unitary eq of dtni on h}
Suppose $T=(T_{1}, \dots ,T_{n})$ is a doubly twisted near-isometry on $\mathcal{H}$. Then $T=(T_{1},\dots ,T_{n})$ on $\mathcal{H}$ is unitarily equivalent to $(M_{T,1},\dots,M_{T,n})$ on $\bigoplus_{A\subseteq I_{n}} H_{\mathcal{D}_{A}}^{2}(\mathbb{D}^{|A|}),$ where 
$$
M_{T,s}=\bigoplus\limits_{A\subseteq I_{n}}M_{A,s}, \ \ \ 1\le s\le n, 
$$ 
with $M_{A,s}$ defined as
$$
M_{A,s}(z^{\bs k}\otimes \eta) =
\begin{cases}
z^{k+e_{i}}\otimes \Lambda_{A, {\bs k}+e_{i}}^*T_{a_{i}}\Lambda_{A, \bs k}(\eta) & \text{if } s=a_{i},1\leq i\leq m \\
z^{k}\otimes \Lambda_{A, {\bs k}}^*T_{q_{j}}\Lambda_{A, \bs k}\eta& \text{if } s=q_{j}, 1\leq j\leq n-m
\end{cases}
$$
and $\Lambda_{A, {\bs k}}: \cl D_A\to T_A^{\bs k}(\cl D_A)$ a unitary for each ${\bs k}\in \bb N_0^{|A|}\}$.
\end{theorem}

\begin{proof}
 Fix $A\subseteq I_{n}$. From Theorem \ref{1 th unitary eq of dtni on ha}, there exists a unitary $U_{A}:\mathcal{H}_{A}\to H_{\mathcal{D}_{A}}^{2}(\mathbb{D}^{|A|})$ such that $U_{A}T_{s}|_{\mathcal{H}_{A}}U_{A}^{*}=M_{A,s}$ for all $s\in I_{n},$ where $M_{A, s}$ is of desired form. Then,
 $$
  U=\bigoplus_{A\subseteq I_{n}}U_{A}
 $$
 defines a unitary from $\bigoplus\limits_{A\subseteq I_{n}}\mathcal{H}_{A}$ onto $\bigoplus\limits_{A\subseteq I_{n}}H^{2}_{\mathcal{D}_{A}}(\mathbb{D}^{|A|}).$ Now, since $T_{s}=\bigoplus\limits_{A\subseteq I_{n}}T_{s}|_{\mathcal{H}_{A}}$, we obtain
\begin{equation*}
 U_AT_sU_A^{*}=\left(\bigoplus_{A\subseteq I_{n}}U_{A}\right) \left(\bigoplus_{A\subseteq I_{n}}T_{s}|_{\mathcal{H}_{A}}\right) \left(\bigoplus_{A\subseteq I_{n}}U_{A}\right)^{*}=\bigoplus_{A\subseteq I_{n}}U_{A}T_{s}|_{\mathcal{H}_{A}}U_{A}^{*}.
\end{equation*}
Hence, 
$$ 
U_AT_sU_A^{*}=\bigoplus_{A\subseteq I_{n}}M_{A,s},
$$
which completes the proof.
\end{proof}

Since every isometry is a near-isometry, Theorem \ref{1 th unitary eq of dtni on h}, in particular, gives an analytic model for a doubly twisted isometry. Interestingly,  our model for doubly twisted isometry coincides with the analytic model given by \cite[Theorem 4.4]{jaydebmansirakshit2022} as we explain below.

\begin{cor}
Suppose $T=(T_{1},T_{2},\dots ,T_{n})$ is a doubly twisted isometry on $\mathcal{H}$,
then $(T_{1},T_{2},\dots ,T_{n})$ on $\mathcal{H}$ is unitarily equivalent to $(M_{T,1},\dots ,M_{T,n})$ on $\bigoplus\limits_{A\subseteq I_{n}} H_{\mathcal{D}_{A}}^{2}(\mathbb{D}^{|A|}),$ where $M_{T,s}=\bigoplus\limits_{A\subseteq I_{n}}M_{A,s}$, where 
\begin{equation}\label{model-sar}
M_{A,s}(z^{k}\eta) =
\begin{cases}
z^{k+e_{i}}\otimes  U_{a_{i},a_{1}}^{k_{1}}\dots U_{a_{i}a_{i-1}}^{k_{i-1}}\eta.&\text{if }  s=a_{i}\in A\\
z^{k}\otimes U_{q_{j}a_{1}}^{k_{1}}\dots U_{q_{j}a_{m}}^{k_{m}}T_{q_{j}}(\eta)& \text{if } s=q_{j}, 1\leq j\leq n-m.
\end{cases}
\end{equation}
\end{cor}
\begin{proof} Since $T$ is also a near-isometry, Theorem \ref{1 th unitary eq of dtni on h} yields that $T$ is unitarily equivalent to $(M_{T,1},\dots,M_{T,n})$ on $\bigoplus\limits_{A\subseteq I_{n}}H_{\mathcal{D}_{A}}^{2}(\mathbb{D}^{|A|}),$ where $M_{T,s}=\bigoplus\limits_{A\subseteq I_{n}}M_{A,s}$ and $M_{A,s}$ is 
defined as
\begin{equation}\label{model}
M_{A,s}(z^{\bs k}\otimes \eta) =
\begin{cases}
z^{{\bs k}+e_{i}}\otimes \Lambda_{A, {\bs k}+e_{i}}^*T_{a_{i}}\Lambda_{A, {\bs k}}(\eta) & \text{if } s=a_{i},1\leq i\leq m \\
z^{\bs k}\otimes \Lambda_{A, {\bs k}}^*T_{q_{j}}\Lambda_{A, {\bs k}}\eta& \text{if } s=q_{j}, 1\leq j\leq n-m
\end{cases}.
\end{equation}

It is evident that the two operator-valued multishifts given by (\ref{model-sar}) and (\ref{model}) differ only in the operator weights. We shall show that that our operator weights, for the case of doubly twisted isometry, coincide with the ones mentioned in (\ref{model-sar}).  To be precise, we shall show that 
$$
\Lambda_{A, {{\bs k}+e_i}}^*T_{a_i}\Lambda_{A, {\bs k}}= U_{a_ia_1}^{k_1}\cdots U_{a_{i}a_{i-1}}^{k_{i-1}} \ \ {\rm and} \ \ 
\Lambda_{A, {\bs k}+e_i}^*T_{q_j}\Lambda_{A, {\bs k}}=U_{q_{j}a_{1}}^{k_{1}}\cdots U_{q_{j}a_{m}}^{k_{m}}T_{q_{j}}
$$
for $a_i\in A$ and $q_j\in I_n\setminus A.$ To this end, recall that $\Lambda_{A, {\bs k}}:\cl D_A\to T_{A}^{\bs k}\cl D_A$ is the unitary operator obtained from 
the polar decomposition of the invertible operator $T_{A}^{\bs k}:\cl D_A\to T_A^{\bs k}\cl D_A$. However, each $T_i$ is an isometry. Thus, $T_A^{\bs k}$ is an isometry. Hence, 
by uniqueness of the polar decomposition, $\Lambda_{A, {\bs k}}=T_{A}^{\bs k}$ on $\cl D_A.$ Then, for $a_i\in A$, 
\begin{eqnarray*}
\Lambda_{A, {\bs k}+e_i}^*T_{a_i}\Lambda_{A, {\bs k}} &=&\Lambda_{A, {\bs k}+e_i}^*T_{a_i}T_A^{\bs k}\\
&=&\Lambda_{A, {\bs k}+e_i}^*T_A^{{\bs k}+e_i}U_{a_ia_1}^{k_1}\cdots U_{a_{i}a_{i-1}}^{k_{i-1}}\\
&=&U_{a_ia_1}^{k_1}\cdots U_{a_{i}a_{i-1}}^{k_{i-1}}.
\end{eqnarray*}
and, for $q_j\in I_n\setminus A$,  
\begin{eqnarray*}
\Lambda_{A, {\bs k}}^*T_{q_j}\Lambda_{A, {\bs k}} &=&\Lambda_{A, {\bs k}}^*T_{q_j}T_A^{\bs k}\\
&=&\Lambda_{A, {\bs k}}^*T_A^{{\bs k}}T_{q_j}U_{q_ja_1}^{k_1}\cdots U_{q_ja_m}^{k_m}\\
&=&T_{q_j}U_{q_ja_1}^{k_1}\cdots U_{q_ja_m}^{k_m}.
\end{eqnarray*}
This completes the proof.
\end{proof}

\subsection*{Acknowledgements} We thank the Mathematical Sciences Foundation, Delhi, for providing partial support and facilities necessary to carry out this work. The first-named author acknowledges partial financial support from the Shiv Nadar Institution of Eminence. The second-named author gratefully acknowledges the support of the Council of Scientific and Industrial Research (CSIR), India. The authors thank Professor Jaydeb Sarkar for suggesting the study of these problems in the present setting.

\end{document}